\documentclass[11pt]{article}

\usepackage{amsmath,amsthm,amsfonts,latexsym}
\usepackage{amssymb}
\usepackage[latin1]{inputenc}

\begin{document}

\newtheorem*{theo}{Theorem}
\newtheorem*{pro} {Proposition}
\newtheorem*{cor} {Corollary}
\newtheorem*{lem} {Lemma}
\newtheorem{theorem}{Theorem}[section]
\newtheorem{corollary}[theorem]{Corollary}
\newtheorem{lemma}[theorem]{Lemma}
\newtheorem{proposition}[theorem]{Proposition}
\newtheorem{conjecture}[theorem]{Conjecture}

\theoremstyle{definition}
 \newtheorem{definition}[theorem]{Definition} 
  \newtheorem{example}[theorem]{Example}
   \newtheorem{remark}[theorem]{Remark}
   
\newcommand{\Naturali}{{\mathbb{N}}}
\newcommand{\Reali}{{\mathbb{R}}}
\newcommand{\Complessi}{{\mathbb{C}}}
\newcommand{\Toro}{{\mathbb{T}}}
\newcommand{\Relativi}{{\mathbb{Z}}}
\newcommand{\HH}{\mathfrak H}
\newcommand{\KK}{\mathfrak K}
\newcommand{\LL}{\mathfrak L}
\newcommand{\as}{\ast_{\sigma}}
\newcommand{\tn}{\vert\hspace{-.3mm}\vert\hspace{-.3mm}\vert}
\def\A{{\cal A}}
\def\B{{\cal B}}
\def\E{{\cal E}}
\def\F{{\cal F}}
\def\H{{\cal H}}
\def\K{{\cal K}}
\def\L{{\cal L}}
\def\N{{\cal N}}
\def\M{{\cal M}}
\def\gM{{\frak M}}
\def\O{{\cal O}}
\def\P{{\cal P}}
\def\S{{\cal S}}
\def\T{{\cal T}}
\def\U{{\cal U}}
\def\V{{\mathcal V}}
\def\qed{\hfill$\square$}
\def\veps{{\varepsilon}}

\title{Heat properties for groups}

\author{Erik B\'edos, Roberto Conti\\}
\date{\today}
\maketitle
\markboth{R. Conti, Erik B\'edos}{
}
\renewcommand{\sectionmark}[1]{}
\begin{abstract}
We revisit Fourier's approach to solve the heat equation on the circle in the context of (twisted) reduced group C*-algebras, convergence of Fourier series and semigroups associated to negative definite functions. We introduce some heat properties for countably infinite groups and investigate when they are satisfied. Kazhdan's property (T) is an obstruction to the weakest property, and our findings 
leave open the possibility that this might be the only one.  On the other hand, many groups with the Haagerup property satisfy the strongest version.  We show that this heat property implies that the associated heat problem has a unique solution regardless of the choice of the initial datum.
\end{abstract}



\section{Introduction}
About two hundred years have passed since J.B.~Fourier's treatise ``Th\' eorie analytique de la chaleur'' (``Analytical theory of heat'') was published in Paris. Few books, if any, have had such a fundamental impact on the development of mathematics, with far reaching consequences, both in theory and applications.   Nowadays, in every undergraduate course on partial differential equations, one discusses how to solve the heat equation  $\partial_t u = \partial_{xx} u $ on the circle (i.e., with periodic boundary conditions) by using Fourier series. One important aspect if one wants to give a rigorous proof is that the potential solution $u(x, t)$ obtained after applying the heat kernel to the initial temperature $f_0$ has a \emph{uniformly convergent} Fourier series in $x$ for each fixed time  $t>0$ and regardless of the regularity properties of $f_0$. This makes it possible to handle the technical issues that arise when proving that in fact one has put his hands on a genuine solution. 
  
Nowadays, a large part of recent research in operator algebras, noncommutative harmonic analysis and noncommutative geometry deals with the study of various properties of the reduced group $C^*$-algebra $C^*_r(G)$ of a discrete group $G$, or its twisted version $C^*_r(G,\sigma)$, where $\sigma$ is a $\Toro$-valued 2-cocycle on $G$. For the sake of clarity, we will only consider the untwisted case in this introduction. But to gain some generality, we will incorporate a 2-cocycle in most of our discussion in the subsequent sections. Recall that if $G$ is a discrete abelian group then, by Fourier transform, $C^*_r(G)$  may be identified with $C(\widehat{G})$, the continuous complex functions on the Pontryagin dual $\widehat G$ of $G$. In the case of the classical heat equation on the circle, one has $G = \Relativi$ and $\widehat{G} = \Toro$, and the solution at any fixed time $t>0$ is a smooth function of the space variable $ x $ in $\Toro$ (hence it has a uniformly convergent Fourier series in $x$). From a modern point of view, a nice feature of the heat kernel is that the assignment $m \mapsto e^{-tm^2}$ is a positive definite function on $\Relativi$ for every $t>0$, which by Schoenberg's theorem is equivalent to the fact that the function   $m\mapsto m^2$ is negative definite on $\Relativi$.  One may now look for suitable reformulations of the heat equation directly in terms of $C^*_r(G)$, thus being meaningful for every (possibly noncommutative) discrete group $G$. In the passage from the classical to the general setting, given a (normalized) negative definite\footnote{We follow \cite{BCR}. Negative definite functions on groups are sometimes called conditionally negative definite functions, or functions (conditionally) of negative type.} function $d:G\to [0, \infty)$, the ordinary Laplacian $\Delta$ has to be replaced with an operator $H^\mathcal{C}_d$ defined on a suitable domain $\mathcal{C}$ in $C_r^*(G)$ and satisfying that $H^\mathcal{C}_d(\lambda(g)) = -d(g) \lambda(g) $ for all $g\in G$, where $\lambda$ denotes the left regular representation of $G$ on $\ell^2(G)$.  An important tool is provided by the theory of multipliers on $G$, i.e., of those functions $\varphi: G\to \Complessi$ satisfying that there exists a bounded linear operator $M_\varphi : C^*_r(G) \to C_r^*(G)$ such that $M_\varphi(\lambda(g)) = \varphi(g) \lambda(g)$ for all $g \in G$. As shown by Haagerup \cite{Haa1}, every positive definite function $\varphi$ on $G$ is a multiplier satisfying that the associated operator $M_\varphi$ is completely positive. In the context of classical Fourier series, multipliers are often called Fourier multipliers and employed for instance in the Fej\'er and the Abel-Poisson summation processes. Needless to say, Fourier series are available also in the setting of reduced group $C^*$-algebras, although convergence w.r.t.~the operator norm for such series have been much less investigated than in the classical situation, one relevant reference here being \cite{BeCo1}.

With this picture in mind, our starting point for the present paper has been to examine to which extent Fourier's approach to solve the heat equation on the circle can be made rigorous when the group $\Relativi$ is replaced by a countably infinite group $G$. More specifically, given a (normalized) negative definite function $d:G\to [0, \infty)$, we consider the natural one-parameter semigroup (``time evolution'') of completely positive maps $\{M_t^d\}_{t\geq 0}$ on $C_r^*(G)$ 
associated to the positive definite functions $e^{-td}$ on $G$, which we apply to some initial datum $x_0$ in $C_r^*(G)$. We then investigate the possibility of expanding the output $u(t)= M_t^d(x_0)$ at each positive time into an operator norm-convergent Fourier series. In doing so, we were somehow inspired by recent advances in the theory of Markov semigroups, noncommutative Dirichlet forms (``energy integrals'') and potential theory (see e.g~\cite{CS, CS2, Cip} and references therein), and to a less extent by a wealth of ideas and techniques around noncommutative geometry and the so-called Baum-Connes conjecture (see \cite{Connes, CCJJV, Val, Laf} as a sample). We recall that the quadratic form $Q(\xi, \eta) := \sum_{g \in G} \overline{\xi(g)} \eta(g)\, d(g)$ on $\ell^2(G)$ turns out to be a Dirichlet form on $C^*_r(G)$ w.r.t.~the canonical trace, as the one-parameter semigroup $\{M^d_t\}_{t\geq 0}$ can be shown to satisfy the required Markovian property. Notice that $\ell := \sqrt{d}$ is then a length function which is also negative definite. If $d$ is proper, and $D_\ell$ denotes the (unbounded) multiplication operator by $\ell$ on $\ell^2(G)$, then $D_\ell$ is an example of a Dirac operator in the framework of Connes' noncommutative geometry \cite{Con2, Connes}. Moreover, the energy function $E(\xi) := Q(\xi, \xi)$ is nothing but $ \langle D_\ell\xi,D_\ell\xi\rangle = \|D_\ell \xi\|^2 = \langle\xi, (D_\ell)^2 \xi\rangle $, i.e., the corresponding expected value at $\xi$ of the associated (unbounded) operator on $\ell^2(G)$ associated with multiplication by $d=\ell^2$.

Now, consider again $x_0 \in C^*_r(G)$ and $u(t) = M^d_t(x_0) \in C^*_r(G)$ for $t\geq 0$. While the uniform convergence of the Fourier series of a function in $C(\Toro)$ is traditionally understood in the conditional sense, it will be more appropriate for us, as in \cite{BeCo1}, to understand convergence of Fourier series in $C^*_r(G)$ as unordered convergence (i.e., unconditional convergence, see Proposition \ref{bdprop})  w.r.t.~operator norm. We let $CF(G)$ denote the subspace of $C_r^*(G)$ consisting of all operators having such a convergent Fourier series. When the initial datum $x_0$ belongs to $CF(G)$, it is not difficult to show that $u(t)$ still belongs to $CF(G)$ at any subsequent time $t>0$. However, the situation becomes more complicated when $x_0 \notin CF(G)$. In this case, it is still conceivable that under suitable circumstances the time evolution will produce some additional regularization (or smoothness) so that $u(t) \in CF(G)$ for all $t >0$, or at least for all $t > t_0$ for some $t_0 \geq 0$. 
To see that this does not always happen, let us assume that $G$ has Kazhdan's property (T) \cite{BHV}.Then, as shown by Akemann-Walter \cite{AW}, every negative definite function on $G$ is necessarily bounded, and this may be used to conclude that there is no way for $M_t^d$ to regularize an ``irregular'' initial datum at any later time, cf.~Corollary \ref{propT2}.These observations seem to suggest an interpretation of groups with property (T) as those for which no 
such regularization process can exist, i.e., as those groups for which Fourier original intuition most likely fails. On the positive side, we have that $u(t) \in CF(G)$ for every choice of $x_0$ as long as $e^{-td} \in \ell^2(G)$ for some $t>0$. Also, if $d$ is proper (so $G$ has the Haagerup property \cite{CCJJV}) and $G$ has exponential H-growth w.r.t.~$d$ (as defined in \cite{BeCo1}), then we get  that $u(t) \in CF(G)$ for all $x_0 \in C^*_r(G)$ and all $t > 0$, cf.~Theorem \ref{subexp}.The class of pairs $(G, d)$ satisfying these last assumptions is quite large. It notably contains $({\mathbb F}_k, |\cdot|)$, where $|\cdot|$ denotes  the canonical word length function on the free group ${\mathbb F}_k$ ($k\geq 2$). In general, letting $MCF(G)$ consists of those multipliers $\varphi$ of $G$  such that $M_\varphi$ maps $C_r^*(G)$ into $CF(G)$, a natural and challenging problem is to find conditions ensuring that  $e^{-td}$ belongs to $MCF(G)$ for some $t>0$ (resp.~for all $t>0$). 

In order to put these considerations on solid grounds, we decided to introduce two properties for countably infinite groups among the several options expressing different flavours of the requirement $u(t) = M_t^d(x_0) \in CF(G)$, i.e., whether for some $d$ this  holds for some (resp.~all) $t>0$, and some (resp.~all) $x_0 \in C_r^*(G)\setminus CF(G)$. Namely, the weakest and the strongest possible choices, called  the {\it weak heat property} and the {\it heat property}, respectively. Groups with property (T) never have the  weak heat property. The converse is open, but we are able to show that a multitude of groups without property (T) do have the weak heat property. Most of these groups are actually finitely generated. Infinitely generated groups seem harder to handle and, for instance, we don't know if $S_\infty$ (the group of finite permutations of $\Naturali$) has the weak heat property. Interestingly, familiar objects like Poincar\'e exponents and Sidon sets play a r\^ole in our investigation of the weak heat property. Although the heat property is  much stronger, quite many groups do have it, e.g., $\Relativi^n$ ($n \geq 1$), finitely generated groups with polynomial growth, non-abelian free groups, and infinite Coxeter groups.  All these groups have the Haagerup property, but the exact relationship between the heat property and the Haagerup property needs to be clarified.  

Next, as an analogue of the Laplacian, we consider the operator $H^\mathcal{C}_d:\mathcal{C} \to C_r^*(G)$ defined by \[H^\mathcal{C}_d(x) = -\sum_{g\in G} d(g) \widehat{x}(g)\lambda(g)\] on the domain $\mathcal{C}= \{ x\in C_r^*(G): \sum_{g\in G} d(g) \widehat{x}(g)\lambda(g) \, \text{is convergent in operator norm}\}$.Then we show in Theorem \ref{heat-solve} that Fourier's intuition was indeed a good one as long as $G$ has the heat property (with respect to~$d$): for every $x_0\in C_r^*(G)$, the natural heat problem $u'(t)= H^\mathcal{C}_d(u(t))$ with initial datum $x_0$ has then a unique solution given by $u(t) = M_t^d(x_0)$ for every $t\geq 0$. Somewhat surprisingly, our proof is not just a simple adaption of the classical argument, but requires some modern tools.

One of our main motivations for writing this article has been to point out a departure from the classical situation.  Namely for a group with property (T) the tentative solution of the heat equation associated to a negative definite function cannot be expanded into an operator norm-convergent Fourier series unless the initial datum itself admits such an expansion. Whether this fact could eventually lead to a characterization of groups with property (T) solely in terms 
of convergence properties of Fourier series remains a challenging problem. We hope this, as well as many other questions scattered in the main body of our paper, will stimulate further interest and investigations.

\bigskip
Our paper is organized as follows. Section 2 contains some preliminaries about negative definite functions on groups, convergence aspects of series in Banach spaces, multipliers and reduced twisted group $C^*$-algebras, decay and growth properties of groups. In the third section we introduce our heat properties, cf.~Definition \ref{weakheat} and Definition \ref{heatproperty}, show that property (T) is an obstruction to the weak heat property in Corollary \ref{propT2}, discuss some general facts about these notions and illustrate both by several examples. Section 4 is devoted to the study of an analogue of the heat equation on a reduced twisted group $C^*$-algebra associated with some negative definite function, where we demonstrate in Theorem \ref{heat-solve} the usefulness of the heat property to guarantee the existence of a solution, regardless of the initial datum. In the final section, we have gathered some further questions and comments, including some indications of possible extensions of our work.

\section{Preliminaries}
Throughout this paper, we assume that $G$ is a countably infinite group with identity $e$, unless otherwise specified.
\subsection{} \label{prelimgr}
We recall that a function $d: G \to \Complessi$ is called \emph{negative definite} if $d(g^{-1}) = \overline{d(g)}$ for all $g \in G$ and, for any $n \in {\mathbb N}$, $g_1,\ldots,g_n \in G$ and $b_1,\ldots,b_n \in \Complessi$ with $\sum_{i=1}^n b_i = 0$, we have
\[\sum_{i,j=1}^n \overline{b_i}b_j \,d(g_i^{-1}g_j) \leq 0\,.\]
Such a function $d$ is said to be \emph{normalized} when $d(e) = 0$. We set
  \[ND_0^+(G) := \big\{d : G \to [0, \infty) : d \text{ is negative definite and }  d(e) = 0\big\}.\]
 As is well known, $ND_0^+(G)$ can be used to characterize some important properties of groups:  
 \begin{itemize}
 \item $G$ has (Kazhdan's) property (T) if and only if every $d\in ND_0^+(G)$ is bounded, see e.g.~\cite{AW, HV, BHV}. 
\item  $G$ has the Haagerup property if and only if there exists some $d\in ND_0^+(G)$ which is \emph{proper}\footnote{A map $c : G\to [0, \infty)$ is said to be proper if the set $\{ g \in G: c(g) \leq t\} $ is finite for every $t\in[0, \infty)$.}, see \cite{CCJJV}. 
\end{itemize}

We will also often consider length functions on $G$: a function $\ell:G \to [0, \infty)$ is called a \emph{length} function  whenever $\ell(e)= 0, \ell(g^{-1}) = \ell(g)$ and $\ell(gh)\leq \ell(g)+\ell(h)$ for all $g, h \in G$. We note that $ND_0^+(G)$ always contains length functions. Indeed, if $d \in ND^+_0(G)$, then $d^{1/2}\in ND_0^+(G)$ and $d^{1/2}$ is a length function on $G$ (cf.~the proof of Proposition 3.3 in \cite{BCR}). 

Another feature of $ND_0^+(G)$ is that it is connected to the representation theory of $G$ via the Delorme-Guichardet theorem: $d \in ND_0^+(G)$ if and only if there exist a unitary (resp.~orthogonal) representation $\pi$ on a complex (resp.~real) Hilbert space $H$ and a $1$-cocycle $b: G \to H$ associated with $\pi$ such that $d(g) = \|b(g)\|^2$ for all $g\in G$. In particular, the maps $g\mapsto \|b(g)\|^2 $ and  $g\mapsto \|b(g)\|$ belong to $ND_0^+(G)$ whenever  $b$ is an homomorphism from $G$ into some Hilbert space (real or complex).

When $G$ acts by isometries on a metric space $(X, \rho)$ and $\ell$ is the (length) function on $G$ given by $\ell(g) = \rho(gx_0, x_0)$ for some choice of base point $x_0\in X$, many authors consider the so-called Poincar\'e exponent of $\ell$.  The definition makes sense for any function $c: G\to [0, \infty)$: we define the \emph{Poincar\'e exponent} $\delta(c) \in [0, \infty]$  of $c$ by 
\[ \delta(c) :=  \inf\big\{s>0 : \sum_{g\in G} e^{-sc(g)} < \infty\}.\]
Poincar\'e exponents have been computed or estimated in many cases, see e.g.~\cite{Patt76, Patterson76, Patt83, Stra04, CTV, AIM, Fra22, MV}. If there exists some $d\in ND_0^+(G)$ such that $\delta(d) < \infty$, then $d$ is proper, hence $G$ has the Haagerup property. Moreover,  if $\delta(d) =0$, then $G$ is amenable. These facts can be deduced from different sources, e.g.~\cite[Theorem 5.3]{GK}, \cite[Theorem 4.1]{CTV} and \cite[Propositions 3.7 and 3.8]{AIM}).

If $c, c': G\to [0, \infty)$, we will write $c \lesssim c'$ whenever there exist some $a, b>0$ and some finite subset $F$ of $G$ such $c(g) \leq a c'(g) + b$ for all $g \in G\setminus F$. The following simple observation will be useful. Assume that $c \lesssim c'$ for some $c: G\to [0, \infty)$ such that $\delta(c) < \infty$. Then an elementary computation gives that $\delta(c') < \infty$: for any $t > a\delta(c)$, we have that
\[\sum_{g\in G\setminus F} e^{-tc'(g)}  \leq \sum_{g\in G} e^{-t (\frac{1}{a}c(g) -b)} =   e^{tb} \sum_{g\in G} e^{- \frac{t}{a}c(g)}  < \infty.\]
Hence we get that $\delta(c') \leq a\delta(c) < \infty$. 

\subsection{} If $S$ is a nonempty set and $\{x_s\} _{s\in S}$ is a family of elements in a normed space $(X, \|\cdot\|)$, 
we will say that the (formal) series $\sum_{s\in S} x_s$ converges to $x \in X$, and write  $\sum_{s\in S} x_s = x$, when the net $\big\{\sum_{s\in F} x_s\big\}_{F\in \mathcal{F}}$ converges to $x$, where $\mathcal{F}$ denotes the set of all finite subsets of $S$, directed by inclusion: in other words, for any $\varepsilon >0$, there exists $F_\varepsilon\in  \mathcal{F}$ such that  $\|x-\sum_{s\in F} x_s \| < \varepsilon$ for every $F\in  \mathcal{F}$ containing $F_\varepsilon$. This notion of convergence is often called unordered convergence. If $X$ is a Banach space, then Cauchy's criterium says that $\sum_{s\in S} x_s$ is convergent if and only for any $\varepsilon >0$, there exists $E_\veps \in  \mathcal{F}$ such that $\| \sum_{s \in F} x_s\| < \veps$ for every $F\in  \mathcal{F}$ which is disjoint from $E_\veps$. Also, $\sum_{s\in S} x_s$ is convergent whenever it is absolutely convergent, i.e., whenever $\sum_{s\in S} \|x_s\|$ is convergent, 
and the converse holds if $X$ is finite-dimensional.

We will repeatedly use the equivalence between conditions $i$) and $v$) in the the following result, connecting unordered convergence and unconditional convergence (see for example~\cite[Section 3.3]{Heil}).
\begin{proposition} \label{bdprop}
Assume that $S$ is a countably infinite set, $X$ is a Banach space, and $\{x_s\} _{s\in S}$ is a family of elements in  $X$. Let $\{s_j\}_{j=1}^\infty$ be any enumeration of $S$ (without repetitions). Then the following statements are equivalent. 

\begin{itemize} 
\item[$i)$] $\sum_{s\in S} x_s$ is convergent.
\item[ii$)$] $\sum_{j\in \Naturali} x_{s_{j}}$ is convergent.
\item[iii$)$] 
$\sum_{j=1}^\infty x_{s_{j}}$ is \emph{unconditionally convergent}, that is, 
  $\sum_{j=1}^\infty x_{s_{\sigma(j)}}=\lim_{n\to \infty} \sum_{j=1}^n x_{s_{\sigma(j)}}$
 exists for every permutation $\sigma$ of $\Naturali$. 
 \item[iv$)$]  $\sum_{j=1}^\infty c_j\, x_{s_j} =\lim_{n\to \infty} \sum_{j=1}^n c_j \,x_{s_j}$ 
 exists for every bounded sequence $\{c_j\}_{j=1}^\infty$ in $\Complessi$. 
\item[v$)$]  $\sum_{s\in S} \varphi(s) x_s$ is  convergent for every $\varphi\in \ell^\infty(S)$.
\end{itemize} 
\end{proposition} 
\begin{proof} 
The equivalence between $i)$ and $ii)$ is an easy exercise. The equivalences between $ii),\,iii)$ and $iv)$ are  shown in \cite[Theorem 3.3]{Heil}. 
The implication $v)\Rightarrow i)$ is obvious. Finally, assume that $i)$ holds and let $\varphi\in \ell^\infty(S)$. To show that $\sum_{s\in S} \varphi(s) x_s$ converges, it suffices to show that $\sum_{j=1}^\infty d_j\, \varphi(s_j)x_{s_{j}}$ exists for every bounded sequence $\{d_j\}_{j=1}^\infty$ in $\Complessi$. But, as the sequence $\{d_j\, \varphi(s_j)\}_{j=1}^\infty$ is bounded  for any such choice of $\{d_j\}_{j=1}^\infty$, this assertion follows from the convergence of the series $\sum_{s\in S} x_s$.
\end{proof} 
\subsection{} 
Our notation and terminology concerning reduced twisted group $C^*$-algebras will essentially be as in \cite{BeCo1}. For the ease of the reader, we recall some of the definitions and results that we will use. Let $\sigma \in Z^2(G, \Toro)$, where $Z^2(G, \Toro)$ denotes the group of normalized $2$-cocycles on $G$ with values in the unit circle $\Toro$. When $\sigma$ is trivial, that is, $\sigma= 1$, we skip it everywhere in our notation. 

The left regular $\sigma$-projective representation $\Lambda_\sigma$ of $G$ on $\ell^2(G)$ is defined by 
 \[ (\Lambda_\sigma(g) \xi) (h) = \sigma(g,g^{-1}h)\, \xi(g^{-1} h), \ \ \xi \in
\ell^2(G),\ g, h \in G.\]
Letting $\{\delta_{h}\}_{h \in G}$ denote the canonical basis of $\ell^2(G)$, we have 
 \[\Lambda_\sigma(g) \delta_{h} = \sigma(g,h) \delta_{gh}, \quad g, h \in G.\]
In particular, $\Lambda_\sigma(g) \delta = \delta_{g}$ for all $g\in G$, where $\delta:= \delta_e$. When $\sigma=1$, we recover the left regular representation of $G$ on $\ell^2(G)$, which we will denote by $\lambda$.
 
As usual, $\B(\ell^2(G))$ denotes the space of all bounded linear operators from $\ell^2(G)$ into itself, and, unless otherwise specified, $\|\cdot\|$ denotes the operator norm on $\B(\ell^2(G))$.The \emph{reduced twisted group $C^*$-algebra} $C_{r}^{*}(G, \sigma)$ (resp.\ the \emph{twisted group von Neumann algebra} $\rm{vN}(G, \sigma)$) is the $C^*$-subalgebra (resp.\ von Neumann subalgebra) of $\B(\ell^2(G))$ generated by $\Lambda_{\sigma}(G)$. In other words,  $C_{r}^{*}(G, \sigma)$ (resp.\ $\rm{vN}(G, \sigma)$) is the closure in the operator norm topology (resp.\ the weak operator topology) of the $*$-algebra $\Complessi(G,\sigma) :=$Span$(\Lambda_{\sigma}(G)).$

The canonical (faithful normal) tracial state $\tau$ on  $\rm{vN}(G,\sigma)$ is given by $\tau(x) = \langle x\delta, \delta\rangle$, and the associated norm on  $\rm{vN}(G,\sigma)$ is defined by $\|x\|_\tau= \tau(x^*x)^{1/2}$. The map  \[x \mapsto \widehat{x}:= x \delta \] is an injective linear map from  $\rm{vN}(G,\sigma)$ into $\ell^2(G)$, satisfying that $\tau(x) = \widehat x(e)$ and $\|x\|_\tau = \|\widehat x\|_2 \leq \|x\|$. Moreover, we have
  \[\widehat{x}(g)= \tau\big(x \Lambda_{\sigma}(g)^*\big), \quad g \in G.\]
 The  \emph{Fourier series} of $ x \in\rm{vN}(G, \sigma)$ is the series $\sum_{g\in G} \widehat{x}(g)\Lambda_{\sigma}(g)$. It always converges to $x$ w.r.t.\  $\|\cdot\|_\tau$. However, when $x \in C_{r}^{*}(G, \sigma)$, this series is not necessarily convergent w.r.t.~operator norm. We therefore set \[CF(G,\sigma):= \Big\{ x \in C^*_{r}(G,\sigma) : \sum_{g\in G} \widehat{x}(g)\Lambda_{\sigma}(g) \ \text{is convergent in operator norm} \Big\}.\]
Although we won't need this fact in the sequel, we mention that $CF(G,\sigma)$ becomes a Banach space w.r.t.~to the norm given by 
\[\|x\|_{\mathcal{F}}:=\sup\Big\{ \,\Big\| \sum_{g\in F} \widehat{x}(g) \Lambda_\sigma(g)\Big\| : F\in \mathcal{F}\Big\}\]
where the set $\mathcal{F}$ consists of all finite subsets of $G$. (This can for example be deduced from \cite[p.~231]{Bac}.)
Moreover, as shown by Bo\.{z}ejko in \cite[Theorem 1.2]{Boz} when $\sigma=1$, the space $CF(G)$ is a Banach $*$-algebra,  and $\|\cdot\|_{\mathcal{F}}$ is equivalent to the norm $\|\cdot\|_U$ on $CF(G)$ given by 
\[ \|x\|_U :=  \Big\| \sum_{g\in G} |\widehat{x}(g)| \,\lambda (g)\Big\|.\]
As a Banach algebra w.r.t.~$\|\cdot\|_U$, $CF(G)$ coincides with the algebra $\mathcal{A}_{\rm max}(G)$ considered by V.~Lafforgue in \cite{Laf}, where it plays an important role in his approach to prove the Baum-Connes conjecture for a certain class of groups, cf.~\cite[Corollaire 0.0.3]{Laf}. This gives us the opportunity to illustrate that using unordered (= unconditional) convergence in the definition of $CF(G)$ is primordial: it is for example known, cf.~\cite{Kah}, that the space \[\big\{ x \in C_r^*(\Relativi) : \lim_{n\to \infty} \sum_{m=-n}^n \widehat{x}(m) \lambda(m)\, \,\text{exists w.r.t.~operator norm}  \big\}\] is not an algebra. 

We also note that Bo\.{z}ejko has shown that $CF(G) \neq C_r^*(G)$ (cf.~\cite[Proposition 3.6]{Boz}). Most probably, it is also true that $CF(G,\sigma)\neq C_r^*(G, \sigma)$ for every $\sigma$ in $Z^2(G, \Toro)$. When $G$ is not periodic, i.e., contains at least one element of infinite order, this can be seen as follows. Since $G$ then contains a copy of $\Relativi$, and every $2$-cocycle on $\Relativi$ is similar to the trivial one, we get that
 $C(\Toro)\simeq C_r^*(\Relativi) \simeq  C_r^*(\Relativi, \sigma)$ may be identified with a $C^*$-subalgebra of $C_r^*(G, \sigma)$. By considering a function in $C(\Toro)$ whose Fourier series is not uniformly convergent, one easily deduces that $C_r^*(G, \sigma)$ contains at least one element not belonging to $CF(G, \sigma)$.  
 
Next, we set \[W(G, \sigma) := \{ x \in C^*_{r}(G,\sigma) : \widehat{x} \in \ell^1(G)\},\] which  is easily seen to be a $*$-subalgebra of $C_r^*(G, \sigma)$ satisfying that 
\[ \Complessi (G, \sigma) \subseteq W(G, \sigma) \subseteq CF(G, \sigma).\]
As is well-known, $W(G, \sigma)$ is a Banach $*$-algebra w.r.t.~to the norm $x\mapsto \|\widehat{x}\|_1$  and $W(G, \sigma)$ is isometrically $*$-isomorphic to the twisted convolution algebra $\ell^1(G, \sigma)$. In particular,  $W(\Relativi)\simeq \ell^1(\Relativi)$ corresponds to the classical Wiener algebra. 
  Another result of Bo\.{z}ejko says that the equality $W(G)= CF(G)$ holds if and only if $G$ is amenable,  cf.~\cite[Proposition 1.3]{Boz}. 
 Whether this equivalence also holds in the twisted case is unclear to us.

Let $\varphi\in \ell^\infty(G)$. We recall that $\varphi$ is called a $\sigma$-\emph{multiplier} on $G$, i.e., $\varphi \in MA(G, \sigma)$, whenever there exists a (necessarily unique) bounded linear operator $M_\varphi: C_r^*(G, \sigma) \to C_r^*(G, \sigma)$ such that
\[M_\varphi(\Lambda_\sigma(g)) = \varphi(g)\Lambda_\sigma(g)\quad \text{for all } g\in G.\]  
 Arguing as Haagerup-de Canni\`ere in their proof of \cite[Proposition 1.2]{DCHa}, one gets that $\varphi  \in MA(G,\sigma)$ if and only if 
there exists a (necessarily unique) normal operator $\widetilde{M}_{\varphi}$ from ${\rm vN}(G, \sigma)$ into itself such that
$\widetilde{M}_\varphi(\Lambda_\sigma(g)) = \varphi(g)\Lambda_\sigma(g)$ for all  $g \in G$, in which case we have 
$\| \widetilde{M}_{\varphi} \| = \| M_{\varphi}\|$.

For $x \in C^*_r(G, \sigma)$ and $g\in G$ we  have that $\widehat{M_\varphi(x)}(g) = \varphi(g)\,\widehat{x}(g)$, so the Fourier series of $M_\varphi(x) $ is given by \[\sum_{g\in G} \varphi(g)\widehat{x}(g)\Lambda_{\sigma}(g).\] 
Moreover, setting
 \[ MCF(G,\sigma):= \big\{ \varphi \in MA(G, \sigma) :  M_{\varphi}  \ \text{maps}\ C^*_{r}(G,\sigma) \ \text{into} \ CF(G,\sigma) \big\}, \]
 we get that \[M_{\varphi}(x)= \sum_{g\in G} \varphi(g)\widehat{x}(g)\Lambda_{\sigma}(g) \ \text{(w.r.t.~operator norm)}\]
 for all $ x \in C^*_{r}(G, \sigma)$ whenever $\varphi \in MCF(G,\sigma)$. 
 
It is easy to see that if $\varphi \in  \ell^2(G)$, then $M_\varphi$ maps $ C^*_{r}(G,\sigma) \ \text{into} \ W(G,\sigma)$, so we have $\ell^2(G) \subseteq MCF(G, \sigma)$, cf.~\cite[p.~356-357]{BeCo1}. One may also produce elements in $MCF(G, \sigma)$ when $G$ satisfies some suitable decay property.
Let  $\kappa: G \to [1,\infty)$. We recall that $G$ is said to be \emph{$\kappa$-decaying} if there exists some $C>0$ such that  
\[ \|x\| \, \leq \,C \,\|\widehat{x}\kappa\|_{2}\]
for all $x\in \Complessi(G)$. Then the following result holds, cf.~\cite[Proposition 4.8]{BeCo1}: 
 \begin{proposition}\label{kappaMCF}
 Assume $G$ is $\kappa$-decaying for some $\kappa: G \to [1,\infty)$. Let $\psi:G\to \Complessi$ be such that  $\psi \kappa$ is bounded.
 Then $\psi \in MCF(G, \sigma)$.
 \end{proposition}
 
\begin{remark}
Clearly,  for $\kappa$ and $\psi$ as in Proposition \ref{kappaMCF}, we have $\psi \in c_0(G)$ whenever $\kappa $ is proper, which will always be the case in our later applications of Proposition \ref{kappaMCF}. We actually don't know if $MCF(G, \sigma)$ is always contained in $c_0(G)$.
 \end{remark}
  
Proposition \ref{kappaMCF} is a strengthening of a result due to Haagerup \cite{Haa1}, where he considers the free group $\mathbb{F}_k$ on $k$ generators for some $k\geq 2$ and  the canonical word length function $L=|\cdot|$ on $\mathbb{F}_k$. After showing that $\mathbb{F}_k$ is $(1+L)^2$-decaying, he proves that any function $\psi: \mathbb{F}_k \to \Complessi$  satisfying that $\psi (1+L)^{2}$ is bounded gives a multiplier on $\mathbb{F}_k$. Proposition \ref{kappaMCF} tells us  that every function $\psi$ in this class  actually belongs  to  $MCF(\mathbb{F}_k)$. As there exist such functions which do not lie in $\ell^2(\mathbb{F}_k)$ (cf.~Example \ref{free-gp2}), we have that $\ell^2(\mathbb{F}_k) \neq MCF(\mathbb{F}_k)$. To our knowledge, not much is known about infinite groups $G$ satisfying that $\ell^2(G) = MCF(G)$ (but see Remark \ref{poincare-exp}).

We will also need the concept of H-growth for groups introduced in \cite{BeCo1}, which relies on the following notion
of size. Let $E$ be a non-empty finite subset of $G$. The {\it Haagerup content} of $E$ is defined as
\[c(E) := \sup\Big\{\|x\| : x \in \Complessi(G),\, {\rm supp}(\widehat{x}) \subseteq E, \ \|\widehat{x}\|_2 \leq 1\Big\} \,.\]
We then have $1 \leq c(E) \leq |E|^{1/2}$, and $c(E) = |E|^{1/2}$ whenever  $G$ is amenable.

Assume $L: G \to [0,+\infty)$ is a proper map such that $L(e) = 0$. Then $G$ is said to have {\it polynomial H-growth} (w.r.t.~$L$)
whenever there exist positive constants $K,p$ such that \[c\big(\{g \in G \ | \ L(g) \leq r\}\big) \leq \, K\, (1+r)^{p}\]
for all $ r \geq 0 $. Similarly, $G$ is said to have {\it subexponential H-growth} (w.r.t.~$L$) if, for any $b>1$, there exists  $r_1\geq 0$ such that
\[c\big(\{g \in G \ | \ L(g) \leq r\}\big) < \,b^r \quad \text{for all }  r \geq r_1.\]
These definitions are analogous to the classical definitions of polynomial/subexponential growth (w.r.t.~$L$), where one uses the cardinality of a set instead of its Haagerup content to measure its size.
 
Note that if $G$ has polynomial (resp.~subexponential) growth (w.r.t.\ $L$), then $G$ has also polynomial (resp.~subexponential) H-growth (w.r.t.~$L$).
Moreover, if  $G$ is amenable, then polynomial (resp.~subexponential) H-growth (w.r.t.\ $L$) reduces to polynomial (resp.~subexponential) growth (w.r.t.\ $L$). 
When $G$ is finitely generated, the function $L$ is frequently chosen to be a word length function w.r.t.~some finite symmetric generator set.
The specific choice of word length function on $G$ is then irrelevant, and usually dropped from the notation. As is well-known, finitely generated groups with subexponential growth are amenable, see for example \cite{DH}.
 
 \section{Some heat properties for groups}

  Throughout this section, we assume that $G$ is a countably infinite group, $\sigma \in Z^2(G, \Toro)$, and $d \in ND_0^+(G)$, unless otherwise specified.
\subsection{The weak heat property}\label{WHP}
The following observation will be useful for our discussion.
 \begin{proposition}\label{CFG2}
Let $\varphi \in MA(G, \sigma)$. Then $M_\varphi$ maps $CF(G,\sigma)$ into itself.
\end{proposition}
\begin{proof}
We may assume that $\varphi \neq 0$. Let $x \in CF(G, \sigma)$. We have to show that the Fourier series of $M_\varphi(x)$
is convergent w.r.t.~operator norm. Since $\varphi$ is bounded, this follows from Proposition 1.1. This can also be proven in an elementary way.

Indeed, let $\varepsilon > 0$ and set $M:= \|M_\varphi\| >0$. Since $x \in CF(G, \sigma)$, Cauchy's criterium gives that there exists a finite set $F_\veps\subseteq G$ such that
\[\Big\| \sum_{g \in F} \widehat{x}(g) \Lambda_\sigma(g)\Big\| < \veps/M\]
for all finite subsets $F$ of $G$ disjoint from $F_\veps$.Thus we get that
\[\Big\|\sum_{g \in F} \varphi(g)\widehat{x}(g) \Lambda_\sigma(g)\Big\| = \Big\| M_\varphi\Big(\sum_{g \in F} \widehat{x}(g) \Lambda_\sigma(g)\Big)\Big\| 
\leq M \, \Big\| \sum_{g \in F} \widehat{x}(g) \Lambda_\sigma(g)\Big\|  < \veps  \]
for all finite subsets $F$ of $G$ disjoint from $F_\veps$.  Cauchy's criterium gives the desired conclusion.
\end{proof}

\begin{remark} \label{ell-inf}
Let $\varphi \in \ell^{\infty}(G)$. If $x \in CF(G, \sigma)$, then Proposition \ref{bdprop} gives that the series $\sum_{g \in G} \varphi(g) \widehat{x}(g) \Lambda_\sigma(g)$ is convergent w.r.t.~operator norm. This means that the assignment $x \mapsto   \sum_{g \in G} \varphi(g) \widehat{x}(g) \Lambda_\sigma(g)$ is a well-defined linear map from $CF(G, \sigma)$ into itself. However, this map can only be extended to a bounded linear map from $C_r^*(G, \sigma)$ into itself when $\varphi$ belongs to $MA(G, \sigma)$, in which case it coincides with $M_\varphi$.
\end{remark} 

Set now $\varphi_t := e^{-t d}$ for every $t\in [0, \infty)$.  By Schoenberg's theorem, 
  each $\varphi_t$ is positive definite, hence belongs to $MA(G, \sigma)$. Moreover, setting $M^d_t := M_{\varphi_t}$ for each $t\geq 0$, 
   $\{M_t^d\}_{t\geq 0}$ is a semigroup of contractions, consisting of unital completely positive maps on $C_r^*(G, \sigma)$ (see~\cite{Haa1}, and  \cite{BeCo1} for the twisted case).

Let $x_0 \in C_r^*(G, \sigma)$ (thinking of it as some initial element). For each $t \geq 0$, set \[x(t) := M^d_t(x_0) \in C_r^*(G,\sigma).\] Then $x(0) = x_0$ and the Fourier series of $x(t)$ is \,$ \sum_{g \in G} e^{-td(g)}\, \widehat{x}(g) \Lambda_\sigma(g)$.
    
 \begin{proposition}\label{CFt1}
 Assume that   $x(t_0) \in CF(G, \sigma)$ for some $t_0 \geq 0$. Then $x(t) \in CF(G, \sigma)$ for all $t\geq t_0$. In particular, if $x_0 = x(0) \in CF(G, \sigma)$, then $x(t) \in CF(G, \sigma)$ for all $t\geq 0$.

\smallskip \noindent 
Moreover, if $0< s\leq t$ and $\varphi_{s} \in MCF(G, \sigma)$, then $\varphi_t \in MCF(G, \sigma)$.
\end{proposition}

\begin{proof}
Let $t \geq t_0$ and set $u= t-t_0$. Then  $x(t)= M^d_u(M^d_{t_0}(x_0)) = M_{\varphi_u}(x(t_0)) \in CF(G, \sigma)$ by Proposition \ref{CFG2}. The last assertion follows in a similar way. 
\end{proof}

Set \[I^{d,\sigma}_{x_0}:=\big\{t \geq 0 \mid x(t) = M_t^d(x_0)\in CF(G, \sigma)\big\} \subseteq [0, \infty).\] Then Proposition \ref{CFt1} gives that  $I^{d,\sigma}_{x_0}=[0, \infty)$ whenever $x_0 \in CF(G, \sigma)$. Moreover, set
 \[\E({d, \sigma}):=\big\{t > 0 \mid \ e^{-td} \in MCF(G, \sigma)\big\}, \, \text{and}\]
  \[\epsilon(d, \sigma):= \inf \E({d, \sigma}) \in [0, \infty].\] 
 When $\sigma =1$, we just write $\epsilon(d)$, which  may be thought of as some new kind of Poincar\'e exponent. Since $\ell^2(G) \subseteq MCF(G, \sigma)$,  we have $ \epsilon(d, \sigma) \leq \delta(d)/2$, with equality when $\ell^2(G)=MCF(G, \sigma)$ (see Remark \ref{poincare-exp} when $\sigma=1$). 
 
If $\E({d, \sigma})$ is non-empty, i.e., $\epsilon(d, \sigma) < \infty$, then  Proposition \ref{CFt1} gives that $\E({d, \sigma})$ is of the form $(\epsilon(d, \sigma), \infty)$ or $[\epsilon(d, \sigma), \infty)$, and it then follows that  $[t, \infty) \subseteq I^{d,\sigma}_{x_0}$ for every $x_0\in C_r^*(G, \sigma)$ whenever $t >\epsilon(d, \sigma)$.
 
  \begin{proposition}\label{propT}
Suppose that $d$ is bounded. Then the following conditions are equivalent:
\begin{itemize}
\item[i$)$] $x_0 \in CF(G, \sigma)$.
\item[ii$)$] $x(t) \in CF(G, \sigma)$ for some $t >0$.
\item[iii$)$] $x(t) \in CF(G, \sigma)$ for all $t \geq 0$.
\end{itemize}
It follows that $I^{d,\sigma}_{x_0} = \emptyset$ for every $x_0 \not\in CF(G, \sigma)$. Hence, $\epsilon(d, \sigma) = \infty$ if  $\sigma =1$  or $G$ is not periodic. 
\end{proposition}
\begin{proof}
The implication $i$) $\Rightarrow iii$) follows  from Proposition \ref{CFG2}, while $iii$) $ \Rightarrow ii$) is trivial. Thus we have to show that $ii$) $ \Rightarrow i$). Assume that $x(t) \in CF(G, \sigma)$ for some $t >0$. To show that $i$) holds, i.e., that the Fourier series $\sum_{g \in G} \widehat{x_0}(g) \Lambda_\sigma(g)$ is convergent w.r.t.~operator norm, it suffices to show that the series
\[\sum_{g\in G} \varphi(g) \,\widehat{x_0}(g) \Lambda_\sigma(g)\] 
is convergent w.r.t.~operator norm for every $\varphi\in \ell^\infty(G)$, cf.~Proposition \ref{bdprop}.

So let $\varphi\in \ell^\infty(G)$. Since $d$ is bounded, it clear that the function $\psi:=\varphi e^{td}$ is bounded too. Now, since $x(t) \in CF(G, \sigma)$ and $\widehat{x(t)}(g) = e^{-td(g)}\, \widehat{x_0}(g)$ for all $g\in G$,  we have that the series \[\sum_{g \in G} e^{-td(g)}\,\widehat{x_0}(g) \Lambda_\sigma(g)\] is convergent w.r.t.~operator norm. Hence, using  again Proposition \ref{bdprop}, we get that the series
\[\sum_{g\in G} \varphi(g) \, \widehat{x_0}(g) \Lambda_\sigma(g) =\sum_{j=1}^{\infty} \psi(g)  \,e^{-td(g)}\,\widehat{x_0}(g) \Lambda_\sigma(g)\]
is convergent w.r.t.~operator norm, as desired.
\end{proof}

\begin{corollary} \label{propT2} Assume that  $G$ has property $($T\,$)$ and let $x_0\in C_r^*(G, \sigma)$. Then,
for any choice of $d$ in $ND_0^+(G)$, we have 
\[I^{d,\sigma}_{x_0} = \begin{cases} [0, \infty) & \text{if} \ x_0 \in CF(G, \sigma), \\ \,\,\, \emptyset & \text{otherwise.}\end{cases}\]
 \end{corollary}

One may legitimately wonder if the converse statement holds. To formulate this in a precise way, we make the following definition.
\begin{definition} \label{weakheat}
We will say that $(G, \sigma)$ has the \emph{weak heat property} whenever there exist some $d\in ND_0^+(G)$, $x_0 \not\in CF(G, \sigma)$ and $t>0$ such that $M_t^d(x_0) \in CF(G, \sigma)$ (i.e., such that $I^{d,\sigma}_{x_0} \neq \emptyset$). When this holds for $\sigma=1$, we just say that $G$ has the \emph{weak heat property}.
\end{definition}

Thus, Corollary \ref{propT2} implies that if $G$ has property (T), then $(G, \sigma)$ does not have the weak heat property. 
 We will later consider a stronger property, called the heat property. We now ask: 
 
 \medskip 
 \emph{If $G$ does not have property $($T\,$)$, does $G$ have the weak heat property?  Also, does then $(G, \sigma)$ have the weak heat property for every $\sigma \in Z^2(G, \Toro)$? } 
 
\medskip It is not difficult to show an $\ell^2$-version of these statements. Indeed, assume that $G$  does not have property (T). Then there exists an unbounded  function $d \in ND_0^+(G)$, cf.~\cite{AW}. So we can pick a  sequence $\{g_n\}_{n=1}^\infty $ of distincts elements in $G$ such that 
 $d(g_n) \geq n$ for every $n\in \Naturali$. For each $t\geq 0$, let $m^d_t:\ell^2(G)\to \ell^2(G)$ denote the bounded operator given by 
 \[[m^d_t(\xi)](g) = e^{-td(g)}\xi(g)\,.\]
 Now, pick  any $c=\{c_n\}_{n=1}^\infty \in \ell^2\setminus \ell^1$, and define $\xi_c \in \ell^2(G)\setminus\ell^1(G)$ by $\xi_c(g) = c_n$ if $g=g_n$ and $\xi_c(g)= 0$ otherwise. Then it is easy to see that  we have
 \[m^{d}_t(\xi_c) \in \ell^1(G) \, \, \text{ for every } t>0\,,\]
 i.e., $m^{d}_t(\xi_c)$ corresponds to an element in $W(G, \sigma) \subseteq CF(G, \sigma)$  for every  $t>0$. 
   
  However, this does not allow us to conclude that $(G, \sigma)$ has the weak heat property. The problem is that it is not clear that  the sequence $c$ can be chosen such that $\xi_c = \widehat{x_0}$ for some $x_0$ in $C_r^*(G, \sigma) \setminus  CF(G, \sigma)$.  Let us assume for simplicity that $\sigma = 1$.
 Setting $E=\{g_n : n\in \Naturali\}$, with the $g_n$'s chosen as above, we would like to find  some $x_0 \in C_r^*(G)$ such that $\widehat{x_0}$ has support in $E$ and  $x_0 \not\in CF(G)$ (i.e., $\widehat{ x_0} \not\in \ell^1(G)$ if $G$ is amenable). In Bo\.{z}ejko's terminology \cite{Boz}, this 
  means that $E$ should \emph{not} be  an unconditional Sidon set (see \cite[Theorem 3.1]{Boz} for a characterization of such sets). We are presently not aware of any study of unconditional Sidon sets in relation with negative definite functions on $G$, even when $G$ is amenable (in which case the notions of unconditional Sidon set and Sidon set agree)\footnote{A subset $E$ of $G$ is called a Sidon set if $x\in W(G)$ for all $x\in C_r^*(G)$ such that $\widehat{x}$ has support in $E$.}.
   
\begin{remark} Consider the following statement for an amenable group $G$:
   
   \medskip \emph{Whenever $E$ is an infinite subset of $G$ such that ${e^{-td}}_{|E} \in \ell^1(E)$ for some $t>0$ and some $d \in ND_0^+(G)$, then $E$ is a Sidon set.}
   
\medskip \noindent In connection with the above discussion, it would be natural to ask whether there exists an amenable group $G$ for which this statement is true. Note that such a group $G$ would then satisfy that $\delta(d) = \infty$ for every $d \in  ND_0^+(G)$ (as $G$ itself is not a Sidon set, since $CF(G) \neq C^*_r(G)$). Moreover, it would  produce a situation where the method outlined above to show that  $G$ has the weak heat property will not work. For the moment it is not clear to us whether groups like $ S_\infty$ (= the group of finite permutations of $\Naturali$) or the first Grigorchuk group \cite{DH} might have the above property. On the other hand, if this should happen, it would give a relatively cheap way to produce Sidon sets, which seems an interesting fact on its own.
   \end{remark}   

 To conclude this subsection, let us recall that a group with property (T) is necessarily finitely generated \cite{BHV}. We will therefore  primarily be interested in finitely generated groups. However, if $G$ is infinitely generated, hence does not have property (T), the question whether $G$ has the weak heat property does still make sense. While the status of the program of characterizing groups with the weak heat property is still uncertain, we will present in the next subsection a long list of examples of such groups (including some infinitely generated ones).  
 
\subsection{Groups with the weak heat property}

\begin{proposition}\label{deltafinite}
Suppose  $G$ is non periodic or that $\sigma=1$. Moreover, assume  that $\delta(d) < \infty$, or, more generally, that $\epsilon(d, \sigma) < \infty$ for some $d\in ND_0^+(G)$. Then $(G, \sigma)$ has the weak heat property.
\end{proposition}
\begin{proof} For any $t > \epsilon(d, \sigma)$, we have $e^{-td} \in MCF(G, \sigma)$, hence $M_t^d(x_0) \in CF(G, \sigma)$ for every $x_0 \in C_r^*(G, \sigma)$. Since we can pick some $x_0 \not\in CF(G, \sigma)$, we obtain that $(G, \sigma)$ has the weak heat property. 
\end{proof} 
Note that $G$ has the Haagerup property whenever $\delta(d)<\infty$, as observed in the preliminaries.

\begin{example} As a very first example, let $G=\Relativi$ and $d\in ND_0^+(\Relativi)$ be given $d(m)=m^2$ for $m \in \Relativi$. Then it is immediate that $\delta(d)= 0$, so $\Relativi$ has the weak heat property. We will see in Example \ref{Zn} that $(\Relativi^n, \sigma)$ has the heat property for every $n\in \Naturali$ and $\sigma \in Z^2(\Relativi^n, \Toro)$.
\end{example} 

\begin{example}
Assume $G$ is a discrete subgroup of $SL(2, \Reali)$, acting on the hyperbolic plane $\mathbb{H}^2=\{ z\in \Complessi : {\rm Im} z >0\}$ by fractional linear transformations, namely \[\begin{pmatrix} a & b \\ c& d\end{pmatrix}\cdot z = \frac{az + b}{cz +d},\] and let $d:G\to [0,\infty)$ be given by $d(g) = \rho( gz_0, z_0)$, where $\rho$ denotes the hyperbolic metric on $\mathbb{H}^2$ and $z_0\in \mathbb{H}^2$. Then it is known that $d$ is negative definite (and proper), cf.~\cite{FH, CCJJV}. Moreover, the Poincar\'e exponent $\delta_G$ of $G$, which is given by
\[\delta_G := \delta(d)= \inf\big\{s>0 : \sum_{g\in G} e^{-sd(g)} < + \infty\big\},\] is independent of the choice of $z_0$ and lies in the interval $[0,1]$, cf.~\cite{Patterson76}. Hence we get from Proposition \ref{deltafinite} that $G$ has the weak heat property.  We will later show by a different method that $SL(2, \Relativi)$ satisfies the heat property, cf.~Example \ref{SL2}.

The same kind of argument as above can be applied to discrete subgroups of $SL(2, \Complessi)$, by considering their action by isometries on the $3$-dimensional hyperbolic space $\mathbb{H}^3$.
\end{example}

\begin{example}
We first recall that if $c \lesssim d$ for some $c: G\to [0, \infty)$ such that $\delta(c) < \infty$, then $\delta(d) < \infty$ (cf.~Subsection \ref{prelimgr}),
so we can conclude from Proposition \ref{deltafinite} that $(G, \sigma)$ has the weak heat property. (Similarly, $(G, \sigma)$ will have the heat property if $\delta(c) = 0$).
  
 A situation where this can be applied is the following. Let $(X, \rho)$ be a metric space and assume that there is a proper action of  a finitely generated group $G$ by isometries on $X$ (proper means that for any compact subset $B$ of $X$, the set $\{ g \in G: gB\cap B \neq \emptyset \} $ is finite). Moreover assume that the orbit space $X/G$ has finite diameter. Let $x \in X$ and set $d(g) = \rho(gx, x)$. Then we have $|\cdot|_S \lesssim d$ for any finite generating subset $S$ of $G$, cf.~\cite[p.~338]{PoSh}. Since $\delta(|\cdot|_S) < \infty$, cf.~Lemma \ref{alg-length}, we see that
 $\delta(d) < \infty$. Moreover, in some cases, e.g.~when $X$ is an $\mathbb R$-tree, it turns out that $d$ is negative definite \cite[Theorem 6.1]{Boz89}. Therefore, when this happens, we get that $G$ has the weak heat property. Examples of groups satisfying all these requirements are fundamental groups of compact manifolds with Euler characteristic less than $-1$, cf.~\cite{PoSh}.
 
 In general, given a finitely generated group $G$, an interesting problem is to produce some $d\in ND_0^+(G)$ such that $|\cdot|_S \lesssim d$ for some (any) word length function $|\cdot|_S$ on $G$. When this can be achieved, we can conclude as above that $G$ has the weak heat property.
\end{example}
  
 \begin{remark}
 It is well-known that amenability of a group implies $\Lambda$-amenability as defined by Lance \cite{La}, which then implies $K$-amenability as defined by Cuntz \cite{Cu}. Moreover, every group with the Haagerup property is $K$-amenable, as shown by Tu (cf.~\cite[Proposition 3.8 and Th\'eor\`eme 9.3]{Tu}), and every infinite $K$-amenable group does not have property (T). It seems therefore natural to ask whether all infinite $K$-amenable groups have the weak heat property. Besides, one may also wonder whether every group with the heat property is $K$-amenable.
 \end{remark}
 
In the next statement we discuss a situation where a group inherits the weak heat property from a subgroup.
\begin{lemma}  \label{subgroup} 
Assume $G$ has a subgroup $H$ which has the weak heat property, and let $d_H \in ND_0^+(H)$, $x_0 \in C_r^*(H) \setminus CF(H)$ and $t>0$ be such that $M_t^{d_H}(x_0) \in CF(H)$. If $d_H$ can be extended to some $d \in ND_0^+(G)$, then $G$ has the weak heat property.
\end{lemma}
\begin{proof} Assume $d_H$ can be extended to $d \in ND_0^+(G)$. Let $\iota$ denote the canonical embedding of $C_r^*(H)$ into $C_r^*(G)$. Then, for $x \in C_r^*(H)$ and $g\in G$, we have
\[\widehat{\iota(x)}(g)=\begin{cases} \widehat{x}(g) & \text{if } g\in H,\\ 0 &  \text{if } g\not\in H.\end{cases}\] 
Thus, if $h\in H$, we have 
 \[\widehat{M_t^d(\iota(x_0))}(h) = e^{-td(h)} \widehat{\iota(x_0)}(h) =  e^{-td_H(h)}\widehat{x_0}(h) = \widehat{M_t^{d_H}(x_0)}(h), \]
 while if $g\not \in H$, we get
 \[\widehat{M_t^d(\iota(x_0))}(g) = e^{-td(g)} \widehat{\iota(x_0)}(g) = 0.\]
 
 Set now $y_0:= \iota(x_0)  \in C_r^*(G)$. Since $\sum_{h\in H} \widehat{M_t^{d_H}(x_0)}(h)\Lambda_H(h)$ is convergent w.r.t.~operator norm in $C_r^*(H)$, we get that the Fourier series of $M_t^d(y_0)$, which is given by
 \[\sum_{g\in G} \widehat{M_t^{d}(y_0)}(g)\Lambda_G(g) =  \sum_{h\in H} \widehat{M_t^{d_H}(x_0)}(h)\Lambda_G(h) = \iota\Big(\sum_{h\in H} \widehat{M_t^{d_H}(x_0)}(h)\Lambda_H(h)\Big),\]
 is convergent w.r.t.~operator norm in $C_r^*(G)$, i.e., $M_t^d(y_0) \in CF(G)$. By the same kind of argument, we also get that $y_0\not \in CF(G)$.  
\end{proof} 
The case where $\sigma$ is not trivial can be handled in a similar way, but we leave this to the reader.

\begin{example} Consider $G = \mathbb{Q}$ and $H= \Relativi$. Then the function $d_\Relativi$ given by $d_\Relativi(m) = m^2$ ($m \in \Relativi$) belongs to $ND_0^+(\Relativi)$ and it has an extension $d \in ND_0^+(\mathbb{Q})$ given by the same formula. Since $\delta (d_\Relativi) = 0$ we can pick $x_0 \in C_r^*(\Relativi) \setminus CF( \Relativi)$ and $t>0$  such that $M_t^{d_ \Relativi}(x_0) \in CF(\Relativi)$. Thus Lemma \ref{subgroup} applies and gives that $\mathbb{Q}$ has the weak heat property.
\end{example} 

\begin{proposition} \label{weakheatfreeprod}
Assume that $G_1*G_2$ is the free product of two nontrivial groups $G_1$ and $G_2$, and that $G_1$ has the weak heat property. Then $G_1 * G_2$ has the weak heat property.
\end{proposition}

\begin{proof}
By assumption, there exist $d_1 \in ND_0^+(G_1)$, $x_1 \in C^*_r(G_1)$ and  $t_1 >0$ as in the definition of weak heat property. Pick any $d_2 \in ND_0^+(G_2)$. Let then  $d$ be the function on $G_1*G_2$ induced by $d_1$ and $d_2$, i.e., $d(e) = 0 $ and \[d(g) = d_{i_1}(s_{1})+ d_{i_2}(s_{2}) + \cdots + d_{i_n}(s_{n})\] whenever  $g= s_{1}s_{2}\ldots s_{n}$, where each $s_{j} \in G_{i_j}$ for some $i_j\in \{1,2\}$, and $i_{j+1}\neq i_j$ for each $j < n$. 
Then $d \in ND_0^+(G_1 * G_2)$. Indeed, let $t>0$. Then each $f_k:= e^{-t d_k}$ is positive definite on $G_k$ ($k=1, 2$). Hence, for $g$ as above, we have 
\[e^{-t d(g)} =  \prod_{j=1}^n e^{-t d_{i_j}(s_{j})} = \prod_{j=1}^n f_{i_j}(s_{j}) =: (f_1*f_2)(g)\] 
Then, by \cite[Theorem 1]{Pic} (see also \cite[Theorem 1]{Boz86} and \cite[Corollary 3.2]{Boz87}), we get that $e^{-t d}= f_1*f_2$ is positive definite on $G$. Since this is true for every $t>0$, Schoenberg's theorem gives that $d$ is negative definite.

Moreover, $d$ coincides with $d_1$ on $G_1$. The conclusion now follows immediately from Lemma \ref{subgroup}.
\end{proof}
This proposition implies that every finitely generated non-abelian free group has the weak heat property. In fact, we will see in Example \ref{free-gp} that they have the heat property.

 \begin{proposition}\label{semidirect}
 Let $G = \Gamma\rtimes K$ be the semidirect product of a group $\Gamma$ by some action of a group $K$, and assume that $K$ has the weak heat property. Then $G$ has the weak heat property.
 \end{proposition}
 \begin{proof}
 By assumption, we can pick $d_K\in ND_0^+(K)$,  $x_0 \in C_r^*(K)\setminus CF(K)$ and $t>0$ such that $M_t^{d_K}(x_0)\in CF(K)$. Define $d$ on $G$ by $d=d_K\circ \pi$ where $\pi: G\to K$ is the canonical map. Then $d$ is an extension of $d_K$ and $d \in ND_0^+(G)$. (Since $d_K$ is unbounded, $d$ is unbounded too, and it is not proper if $\Gamma$ is infinite.) It follows from Lemma \ref{subgroup} that $G$ has the weak heat property. 
 \end{proof}
 
 \begin{remark}\label{semidirect2}
We note that if $K$ has the heat property, then $G$ will have some kind of relative heat property which is stronger than the weak heat property (cf.~Example \ref{SL2}  about $\Relativi^2\rtimes SL(2, \Relativi)$). 
 \end{remark}

\begin{example}
Let $G$ be the discrete $3$-dimensional Heisenberg group. Then it is well-known that $G$ may be written as a semidirect product $\Relativi^2\rtimes \Relativi$. Hence Proposition \ref{semidirect} gives that $G$ has the weak heat property. Now, it is also known that $G$ has polynomial growth, and this will imply that $G$ has the heat property, as we will see in Example \ref{CS}.
\end{example}

\begin{example} Let $G=\Relativi_2 \wr \Relativi $ denote the lamplighter group, i.e., $G$ is the semidirect product $(\bigoplus_\Relativi \Relativi_2) \rtimes \Relativi$, where $\Relativi$ acts by translation on the index set $\Relativi$ (see e.g.~\cite[p.~103]{DH}). Since $\Relativi$ has the weak heat property, we get from Proposition \ref{semidirect} that $G$ has the weak heat property. This argument applies to any wreath product $H\wr K$ as long as $K$ has the weak heat property.  
\end{example}

\begin{example}\label{purebraid}
For  $n \geq 2$, let $P_n$ denote the pure braid group on $n$ strings \cite[Section 1.3]{KT}. Then $P_2 \simeq \Relativi$ and $P_n \simeq \mathbb{F}_{n-1}\rtimes P_{n-1}$ for $n \geq 3$, so we get
\[ P_{n} \simeq  \mathbb{F}_{n-1}\rtimes ( \mathbb{F}_{n-2} \rtimes ( \cdots \rtimes(\mathbb{F}_2\rtimes \Relativi)\cdots ))\]
for $n \geq 3$. Hence, by using Proposition \ref{semidirect} repeatedly, we get that $P_n$ has the weak heat property. 

Similarly, letting $B_n$ denote the braid group on $n$ strings, we will see in Example \ref{braid} that they also have the weak heat property.
\end{example}

\begin{remark} \label{direct}
Proposition \ref{semidirect} may clearly be applied to a direct product $\Gamma \times K$.
Choosing $K$ with the weak heat property, e.g., $K= \Relativi$, we can produce many examples of groups with the weak heat property, including amenable ones, and also infinitely generated ones. If $\Gamma$ has intermediate growth (e.g., the first Grigorchuk group \cite{DH}), we get that $\Gamma \times \Relativi$ has intermediate growth and the weak heat property. Similarly, one can produce amenable groups having exponential growth and the weak heat property. 
\end{remark}

\begin{example} 
Consider $G = SL(3, \Relativi) \times \Relativi$.  Then $G$ has the weak heat property, while its subgroup $SL(3, \Relativi)$ does not (since it has property (T)). This example shows that  the weak heat property does not necessarily pass to subgroups. 
\end{example}

\begin{remark} One may wonder how to deal with the general situation where $G$ has a quotient group  which has the weak heat property. Will $G$ then also have this property? In lack of a positive answer for the moment, one could say that $G$ has the \emph{ultraweak heat property} if $G$ has at least one quotient group with the weak heat property. Obviously, any group with the weak heat property will then have the ultraweak heat property. Note that if $G$ has the ultraweak heat property, then $G$ does not have property (T). Instead of trying to prove that a group without property (T) has the weak heat property, it might turn out to be easier to show that it has the ultraweak heat property. Of course, this will not help to handle simple groups or groups with only finite quotients.
\end{remark}

\begin{example} For $p\geq 2$, let $BS(p, p)$ denote the Baumslag-Solitar group with presentation 
 \[ BS(p,p) =\langle a, b \mid ab^{\,p}= b^{\,p}a\rangle.\]
Then its center $Z(p)$ is the cyclic subgroup $\langle b^{\,p}\rangle$, and $BS(p,p) / Z(p) $ is clearly isomorphic to $K:=\Relativi*\Relativi_p$. So there is a homomorphism $\pi$ from $BS(p, p)$ onto $K$. Letting $\pi': K\to \Relativi$ be the natural surjective homomorphim, we get that $\pi'\circ \pi$ is a homomorphism from $BS(p,p)$ onto $\Relativi$. Since $\Relativi$ is free, this homomorphism splits, so we have $BS(p,p)\simeq H\rtimes \Relativi$, where $H:= \ker(\pi'\circ \pi)$. 
Thus $BS(p,p)$ has the weak heat property. 
 \end{example}
  
Using the same kind of argument as in the example above, we get:
\begin{proposition}
Assume  $G_{\rm ab} := G/[G, G]$ is finitely generated and infinite.\footnote{This assumption implies that $G$ does not have property (T), cf.~\cite{BHV}.} Then $G$ has the weak heat property.
\end{proposition}
\begin{proof}
Since $G_{\rm ab}$ is an infinite finitely generated abelian group, there exists a homomorphism from $G_{\rm ab}$ onto $\Relativi$, so we may write $G\simeq H\rtimes \Relativi$ for some normal subgroup $H$ of $G$. Hence the assertion follows from Proposition \ref{semidirect}.  
\end{proof}
This proposition can be used to give more examples of groups without property (T) having the weak heat property. 

\begin{example}
Let $F$ be the Thompson group with presentation \[ F=\langle a,b :  [ab^{{-1}}, a^{{-1}}ba]=[ab^{{-1}},a^{{-2}}ba^{{2}}]=1 \rangle,\] see e.g.~\cite{CFP}. Since $F_{\rm ab}$ is isomorphic to $\Relativi^2$, we get that $F$ has the weak heat property.
 \end{example}
 
 \begin{example}
Let $G$ be the fundamental group of a closed orientable surface of genus $g\geq 1$.Then it is known that  $G_{\rm ab}$ is isomorphic to $\Relativi^{2g}$, cf.~\cite{DH}, so it follows that $G$ has the weak heat property. A similar argument can be applied in the  non-orientable case if the genus is $\geq 2$.
 \end{example}
 
 \begin{example} \label{braid}
 Let $G= B_n$ be the braid group on $n$ strings ($n\geq 2$). Then $G_{\rm ab}$ is isomorphic to $\Relativi$ (see e.g.~\cite[Section 1.1]{KT}).
 Hence $G$ has the weak heat property.
 \end{example}
 
 \begin{example} Let $B_\infty$ denote the braid groups on infinitely many strands, which may be written as the inductive limit $B_\infty = \underset{\to}{\lim } \, B_n$ where the embedding $B_n\to B_{n+1}$ is determined by adding one string connecting $n+1$ to itself, cf.~\cite{KT}. Then it is easy to see that $B_\infty$ has $\Relativi$ as a quotient group, so it follows that $B_\infty$ has the weak heat property. Letting $P_\infty$ denote the pure braid group on infinitely many strings, i.e., $P_\infty$ is the kernel of the canonical homomorphism from $B_\infty$ onto the infinite symmetric group $S_\infty$, then we also get that $P_\infty$ has $\Relativi$ as a quotient, so it has the weak heat property. 
 \end{example}

\subsection{The heat property for groups} \label{heat-property}
Although we have no definite answer to the question whether every group without property (T) has the weak heat property, this line of thought opens the door to some interesting development.

We recall that $d\in ND_0^+(G)$.  An immediate consequence of the definition of $MCF(G, \sigma)$ is that the following two conditions are equivalent:
\begin{itemize} 
\item[(H1)] $M_t^d(x_0)$ belongs to $CF(G, \sigma)$ for every $x_0 \in C_r^*(G, \sigma)$ and every $t>0$.
\item[(H2)] $\epsilon(d, \sigma) = 0$, that is, $\varphi_t = e^{-td} \in MCF(G, \sigma)$ for all $t>0$.
\end{itemize}

\begin{definition}\label{heatproperty}
We will say that $(G, \sigma)$ has the \emph{heat property w.r.t.~$d$} when any of the equivalent conditions (H1) and (H2) holds. Moreover, we will say that $(G, \sigma)$ has the \emph{heat property} if there exists some  $d\in ND_0^+(G)$ such that $(G, \sigma)$ has the heat property w.r.t.~$d$. If $\sigma=1$, we skip it in our terminology.
\end{definition} 

\begin{remark} \label{summing} Since $e^{-td}$ converges pointwise to $1$ as $t\to 0^+$, and $\|M_t^d\|=1$ for all $t>0$, we have that $M_t^d(x) \to x $ as $t\to 0^+$ for every $x \in C_r^*(G, \sigma)$. Hence,  saying that $(G, \sigma)$ has the heat property w.r.t.~$d$ is equivalent to requiring that the net $\{e^{-td}\}_{t>0}$ is a Fourier summing net for $(G, \sigma)$ as defined in \cite{BeCo1}. 
\end{remark}

Using Proposition \ref{CFt1}, we get that $(G, \sigma)$ has the heat property w.r.t.~$d$ if and only if $I^{d,\sigma}_{x_0} = (0, \infty)$ for every $x_0 \not\in CF(G, \sigma)$. We also note that, by Corollary \ref{propT2}, the heat property fails (for all $\sigma$ and $d$) whenever $G$ has property (T). It is also clear that the heat property implies the weak heat property. 

\begin{remark} \label{bounded}  If $d$ is bounded, then there exists no  $t>0$ such that $\varphi_t=e^{-td}$ belongs to $MCF(G, \sigma)$, at least when $\sigma =1$ or $G$ is non-periodic. Indeed, if such  $t >0$ exists, then we can pick some $x_0\in C_r^*(G, \sigma)\setminus CF(G, \sigma)$ and get that $x(t)=M_t^d(x_0) \in CF(G, \sigma)$, contradicting Proposition \ref{propT}.  
\end{remark}

\begin{remark} Assume that $H$ is an infinite subgroup of $G$ and $(G, \sigma)$ has the heat property. Then $(H, \sigma)$ has the heat property.
Indeed, assume $(G, \sigma)$ has the heat property w.r.t.~$d$, and let $d_H$ denote the restriction of $d$ to $H$. Clearly, $d_H$ is negative definite and normalized. Moreover, $(H, \sigma)$ has the heat property w.r.t.~$d_H$. Indeed, it is not difficult to check that $e^{-td_H}$ belongs to $MCF(H, \sigma)$ for all $t > 0$ (using similar arguments as in Lemma \ref{subgroup}). 

It follows  that a group without property (T) does not necessarily have the heat property. Indeed, $G= SL(3, \Relativi) \times \Relativi$ neither has
property (T) (cf.~\cite[Proposition 1.7.8]{BHV}) nor the heat property (since its subgroup $SL(3, \Relativi)$ has property (T)).  On the other hand, it has the weak heat property (by~Proposition \ref{semidirect}). 
\end{remark}

 As a consequence of Remark \ref{bounded}, we only consider unbounded $d$ in the rest of this subsection.  
One may actually wonder if the heat property (w.r.t.~$d$) implies that $d$ must be proper, i.e., $e^{-t d} \in c_0(G)$ for some (hence all) $t > 0$. This would imply that $G$ must have the Haagerup property, cf.~\cite{CCJJV}. This will be the case in all the examples we are going to exhibit.

 We now turn our attention to providing conditions  for the heat property to hold. We examine two possible situations. The first one is that $\delta(d) = 0$, which is only applicable when $G$ is amenable. The second one requires that $G$ satisfies a suitable decay property, cf.~Proposition \ref{etd-decay}.  

As pointed out in the preliminaries, the first two statements of the following proposition are known. 
\begin{proposition} \label{ell} 
Assume that $\delta(d) = 0$. Then $G$ is amenable, $d$ is proper, and $(G, \sigma)$ has the heat property w.r.t.~$d$.
\end{proposition}
\begin{proof} 
The assumption gives that $\{ \varphi_{t}\}_{t>0}$ is a net of normalized  positive definite functions in $\ell^2(G)$ converging pointwise to $1$ as $t\to 0^+$, so it follows that $G$ is amenable (cf.~the proof of \cite[Theorem 4.1]{CTV}). It is  obvious that $d$ is proper. Finally, for every $t>0$, we have that  $\varphi_{t} \in MCF(G, \sigma)$ (since $\varphi_t \in \ell^2(G)$).
\end{proof}

\begin{example} \label{Zn}
Assume that  $G^d_0:=\{ g \in G \mid d(g)=0\}$ is finite and $\sum_{g\in G\setminus G^d_0} d(g)^{-p} < \infty$ for some $p \in [1, \infty)$. We note that  $\delta(d) = 0$.

Indeed, let $t>0$. Then for every $g \in G\setminus G^d_0$ we have \[ e^{-td(g)} = d(g)^pe^{-td(g)}\, \frac{1}{d(g)^p}\]
As the function $g\mapsto d(g)^pe^{-td(g)}$ is bounded on $G$, say by $M_{t, p}$, we get that 
\[ \sum_{g\in G} e^{-td(g)} \leq \, \big|G^d_0\big| + M_{t, p}\sum_{g\in G\setminus G^d_0} \frac{1}{d(g)^p}  < \infty.\] 
Thus it follows from Proposition \ref{ell} that $G$ is amenable and $(G, \sigma)$ has the heat property w.r.t.~$d$. 

This applies for example to $G=\Relativi^n$ ($n \in \Naturali$) and $d=|\cdot|_1$, $d= |\cdot|_2$ or $d=|\cdot|_2^2$ (with a suitably chosen value of $p$ in each case), where $|m|_1:= |m_1| + \cdots +|m_n|$ and $|m|_2 := (m_1^2 + \cdots +m_n^2)^{1/2}$ for $m = (m_1, \ldots, m_n) \in \Relativi^n$. It is well-known that $d\in ND_0^+(\Relativi^n)$ for all these choices (see for example \cite[Theorem 5.7]{BeCo1} and references therein).

 Up to similarity, any $\sigma \in Z^2(\Relativi^n, \Toro)$ is of the form $\sigma(m,m')  = e^{ i \, m^t\Theta m'}$ for $m,m'\in \Relativi^n$ for some skew-symmetric $n\times n$ matrix $\Theta$ with coefficients in $[0, 2\pi)$.
 When $\Theta = 0$, we have $C_r^*(\Relativi^n, \sigma) \simeq C(\Toro^n)$, and if $d=|\cdot|_2^2$, this is precisely the situation one meets in connection with the classical heat equation on $\Toro^n$, i.e., on boxes in $\Reali^n$ with periodic boundary conditions. If $\Theta \neq 0$, then  $C_r^*(\Relativi^n, \sigma)$ is the noncommutative $n$-torus $\mathcal{A}_\Theta$ associated to $\Theta$. 
\end{example}

Similar examples arise when considering other finitely generated groups.
\begin{lemma}\label{alg-length}
 Assume $G$ is finitely generated and $S$ is a (finite) symmetric generator set for $G$. Let $L=|\cdot|_S$ be the associated word length function on $G$. 
\begin{itemize}
\item[$a)$] Set $t_0 :=\ln(|S|)/2$. Then $e^{-tL}$ belongs to $\ell^2(G)$ whenever $t > t_0$. Hence, $\delta(L) <\infty$. 
\item[$b)$] Assume $G$ has subexponential growth. Then $e^{-tL}$ lies in $\ell^2(G)$ for all $t>0$. Hence, $\delta(L)=0$. 
\end{itemize}
\end{lemma}
\begin{proof}
$a$) Set $L_n=\{ g \in G : L(g) = n\}$ for each $n\in \Naturali$. Then, for any $t>0$, we have 
\[ \sum_{g\in G} \big(e^{-tL(g)}\big)^2 =  \sum_{n=1}^\infty |L_n| \, r^n\,,\] 
where $r:=e^{-2t} \in (0, 1)$. The radius of convergence $R$ of this power series in $r$ is given by $R^{-1}= \limsup_n |L_n|^{1/n}$. 
Since $|L_n| \leq |S|^n$ for every $n$, we get $R \geq |S|^{-1}$, so the series above is always convergent whenever $e^{-2t} < |S|^{-1}$, i.e., $t >  t_0:=\ln(|S|)/2$. Setting $s=2t$, we get that $\sum_{g\in G} e^{-sL(g)} < \infty$ whenever $s > \ln(|S|)$, so $\delta(L) < \infty$.

\smallskip 
$b$) The assumption gives that for any $b>1$, there exists $n_0 \in \Naturali$ such that $|L_n| < b^n$ for every $n\geq n_0$. This clearly implies that $R^{-1} \leq b$ for every $b>1$, hence that $R \geq 1$. It follows that $e^{-tL}$ lies in $\ell^2(G)$ for all $t>0$, hence that $\delta(L)=0$.
\end{proof}

\begin{proposition} \label{alg-length2} 
Let $G, S$ and $L=|\cdot|_S$ be as in the previous lemma, and let $x_0\in C_r^*(G, \sigma)$.
\begin{itemize}
\item[$i)$] Assume that $L$ is negative definite and set $x(t)=M^L_t(x_0), t\geq 0$. 
Then we have that $x(t) \in CF(G, \sigma)$ for all $t > \ln(|S|)/2$.

Moreover, if $G$ has subexponential growth, then $(G, \sigma)$ has the heat property w.r.t.~$L$. 

\item[$ii)$] Assume that $L^2$ is negative definite. Then $G$ is amenable and $(G, \sigma)$ has the heat property w.r.t.~$L^2$. 
\end{itemize}
\end{proposition}
\begin{proof} $i$) Let $t > \ln(|S|)/2$. Then  part a) of Lemma \ref{alg-length} gives that $e^{-tL} \in \ell^2(G)$, hence that  $e^{-tL} \in MCF(G, \sigma)$. Thus, $x(t) = M_{e^{-tL}}(x_0) \in CF(G, \sigma)$.

Similarly, part b) of Lemma \ref{alg-length} in combination with Proposition \ref{ell} gives the second statement. 

$ii$) Since  $e^{-tL^2}$ lies in $\ell^1(G)$, hence in $\ell^2(G)$, for every $t>0$ (cf.~the proof of \cite[Proposition 24]{Con2}),  we can apply Proposition \ref{ell} with $d= L^2$.
\end{proof}

A typical situation where both $i$) and $ii$) in  Proposition \ref{alg-length2} can be applied is when $G= \Relativi$ and $L(m) = |m|$.  
In general, if $L$ is a word length function on a finitely generated group, it is not clear under which conditions $L$, or $L^2$ (if $G$ is amenable), is negative definite.
 
\begin{remark}
Assume $G$ is amenable. If $G$ is finitely generated, it is natural to wonder whether one can always find a negative definite word length function $L$ on $G$, so that Proposition \ref{alg-length2} can be applied. Of course, a more general question is whether there always exists some $d\in ND_0^+(G)$ such that $(G, \sigma)$ has the heat property w.r.t.~$d$. When $G$ is  finitely generated and has polynomial growth w.r.t.~some (hence any) word length function, it follows from a result of Cipriani and Sauvageot \cite{CS} that the answer to this question is positive, cf.~our discussion in Example \ref{CS}.  The square root of the function $d$ constructed in their paper is then a proper negative definite length function (but not a word length).
\end{remark}
 
\begin{example}\label{free-gp}
Let $ \mathbb{F}_k=\langle a_1, \ldots, a_k\rangle$ be a non-abelian free group ($k\geq 2$), and let $L=|\cdot|$ denote the canonical word length function on $\mathbb{F}_k$ associated to $S=\{a_1^{\pm 1}, \ldots, a_k^{\pm 1}\}$, which is negative definite \cite{Haa1}. As every  $\sigma \in Z^2(\mathbb{F}_k, \Toro)$ is similar to $1$, we may assume that $\sigma=1$. Then Lemma \ref{alg-length} tells us that $e^{-tL}\in \ell^2(\mathbb{F}_k)$ whenever $t > \ln(2k)/2$. In fact, since $\big|\{g \in \mathbb{F}_k : |g|= n\}\big| = 2k(2k-1)^{n-1}$, a  look at the proof of this lemma gives that  $e^{-tL}\in \ell^2(\mathbb{F}_k)$ if and only if $t > \ln(2k-1)/2$. Proposition \ref{alg-length2} gives that $x(t)=M^L_t(x_0) \in CF(G)$ for every $x_0\in C_r^*(\mathbb{F}_k)$ and every $t > \ln(2k-1)/2$. However, we will see in see Example \ref{free-gp2} that $\epsilon(|\cdot|) = 0$, i.e., that $\mathbb{F}_k$ has the heat property w.r.t. $|\cdot|$. 
\end{example}

\begin{remark} \label{poincare-exp}
One may wonder if there exists some (nonabelian) group $G$ such that
\begin{equation} \label{MCF_L2} MCF(G)= \ell^2(G).
\end{equation}
When $G$ is amenable, using Bojzeko's result that $CF(G) = W(G)$, it would suffice to show that if $\varphi \in MA(G)$ is such that $\varphi\widehat{x} \in \ell^1(G)$ for all $x\in {\rm vN}(G)$, then $\varphi \in \ell^2(G)$. 

For $G$ abelian, it is in fact true that the equality (\ref{MCF_L2}) holds. Indeed, as shown in  \cite[Theorem 2.1]{Edw} (see also~\cite[Theorem 2]{Helg}), if $\widehat{G}$ denotes the dual group of $G$ and $\varphi:G\to \Complessi$ is such that  $\varphi\widehat{f} \in \ell^1(G)$ for all $f\in C(\widehat{G})$, then $\varphi \in \ell^2(G)$. Since $C_r^*(G)$ is canonically isomorphic to $C(\widehat{G})$, the assertion follows. It is conceivable that (\ref{MCF_L2}) holds for every amenable group, but for the time being we will not delve any further into this topic.   

We note that if (\ref{MCF_L2}) holds, then $\epsilon(d) = \delta(d)/2$, so
$G$ has the heat property w.r.t.~$d$ if and only if $\delta(d) = 0$, and that if $\delta(d) > 0$, then we have \[\Big(\frac{\delta(d)}{2}, \infty\Big) \subseteq I^{d}_{x_0} \subseteq \Big[\frac{\delta(d)}{2}, \infty\Big)\] for every $x_0 \not \in CF(G)$. 
\end{remark}

To handle the case where $\delta(d) > 0$,  and include nonamenable groups in our discussion, we will exploit the decay behavior of groups as studied in \cite{BeCo1}.

\begin{proposition} \label{kappaMCF2}
Assume $(G, \sigma)$ has $\kappa$-decay for some $\kappa:G\to [1, \infty)$ satisfying that $\kappa \leq C e^{t_0d}$ for some $C>0$ and some $t_0>0$. Then $\varphi_t = e^{-td} \in MCF(G, \sigma)$ and $x(t)=M_t^d(x_0) \in CF(G, \sigma)$  for every $t \geq t_0$ and every $x_0\in C_r^*(G, \sigma)$. 
\end{proposition}

\begin{proof} 
As $\varphi_{t_0}\kappa = e^{-t_0d}\kappa \leq C$, we get from Proposition \ref{kappaMCF} that $\varphi_{t_0} \in MCF(G, \sigma)$.  This implies that $x(t_0) = M_{t_0}^d(x_0) \in CF(G, \sigma)$ for every $x_0\in C_r^*(G, \sigma)$. It follows then from Proposition \ref{CFt1} that both assertions hold for any $t \geq t_0$. 
\end{proof}

\begin{proposition} \label{etd-decay}
Suppose that  $G$ is $e^{td}$-decaying for every $t>0$. Then $\epsilon(d, \sigma) = 0$, i.e., $(G, \sigma)$ has the heat property w.r.t.~$d$.
\end{proposition}
\begin{proof}
Given any $t_0>0$, setting  $\kappa = e^{t_0d}$ and  $C=1$, we can apply Proposition \ref {kappaMCF2} to obtain the desired conclusion.
\end{proof} 

\begin{remark} If we only assume that $G$ is $e^{t_0d}$-decaying for some $t_0>0$, then we get, essentially as in the above proof, that $(G, \sigma)$ has the weak heat property. 
\end{remark}

\begin{remark} 
Proposition \ref{etd-decay} can easily be generalized as follows. Assume that for each $t>0$, $(G, \sigma)$ has $\kappa_t$-decay for some function $\kappa_t:G\to [1, \infty)$ satisfying that $\kappa_t\, e^{-td}$ is bounded. Then $(G, \sigma)$ has the heat property w.r.t.~$d$. 
\end{remark}

The concept of  H-growth was introduced in \cite{BeCo1} as a useful replacement of the usual notion of growth for nonamenable groups.  The relevance of subexponential H-growth is illustrated here by the following result.

\begin{theorem} \label{subexp}
Assume that $d$ is proper and $G$ has subexponential H-growth w.r.t.~$d$. 
Then $\epsilon(d, \sigma) = 0$, i.e., $(G, \sigma)$ has the heat property w.r.t.~$d$.
\end{theorem}
\begin{proof}
Using Remark \ref{summing}, the result follows from the last statement in  \cite[Theorem 5.9]{BeCo1}
Alternatively, we have that $G$ is $e^{t d}$-decaying for all $t>0$, cf.~\cite[Theorem 3.13, part 2)]{BeCo1}, 
so  the conclusion follows from Proposition \ref{etd-decay}.
\end{proof}

\begin{example} \label{free-gp2}
Let's go back to Example \ref{free-gp} dealing with a non-abelian free group $\mathbb{F}_k$ $ (2\leq k < \infty$) and its canonical word length function $L=|\cdot|$, as considered by Haagerup in his seminal paper \cite{Haa1}. Since $\mathbb{F}_k$ has subexponential H-growth (in fact,  polynomial H-growth) w.r.t.~$|\cdot|$, cf.~\cite[Example 3.12]{BeCo1}, we may apply Theorem \ref{subexp} and obtain that $\epsilon(|\cdot|) = 0$, i.e., $\mathbb{F}_k$ has the heat property w.r.t.~$|\cdot|$.This means that for any $x_0 \in C_r^*(\mathbb{F}_k)$, the Fourier series 
\begin{equation}\label{free-sol} \sum_{g\in \mathbb{F}_k} e^{-t|g|} \,\widehat{x_0}(g) \,\lambda(g)\,,\end{equation}
converges in operator norm to $x(t)=M_t^{L}(x_0)$ for {\it every} $t>0$, even if the Fourier series of $x_0$ is not convergent in operator norm. 
It is worth pointing out again that the function $g \mapsto e^{-t|g|}$ belongs to $\ell^2(\mathbb{F}_k)$ only when $t$ is sufficiently large (cf.\ Example \ref{free-gp}).  
\end{example}

\begin{example}
\label{Cox}
Let $G$ be an infinite Coxeter group with a finite generator set $S$ such that $(G, S)$ is a Coxeter system, and let $L$ be the word length function on $G$ associated to $S$. Then $L$ is negative definite and $G$ has polynomial H-growth w.r.t.~$L$, cf.~\cite[Example 5.10]{BeCo1} and references therein. Thus
Theorem \ref{subexp} gives that $G$ has the heat property w.r.t.~$L$.
\end{example}

In view of Example \ref{free-gp2}, it is natural to look at the case of free products. 

\begin{proposition} \label{freeprod}
Assume that $G_1*G_2$ is the free product of two nontrivial groups $G_1$ and $G_2$, and that each $G_j$ has polynomial H-growth w.r.t.~some proper, negative definite, integer-valued length function $L_j$ satisfying $\{ g \in G_j : L_j(g)=0\} =\{e\}$. 

Let then  $L$ be the  length function on $G_1*G_2$ induced by $L_1$ and $L_2$, i.e., $L(e) = 0 $ and \[L(g) = L_{i_1}(s_{1})+ L_{i_2}(s_{2}) + \cdots + L_{i_n}(s_{n})\] whenever  $g= s_{1}s_{2}\ldots s_{n}$, where each $s_{j} \in G_{i_j}$ for some $i_j\in \{1,2\}$, and $i_{j+1}\neq i_j$ for each $j < n$. 

Then $L$ is negative definite and $(G_1*G_2, \sigma)$ has the heat property w.r.t.~$L$ for every $\sigma$ in $Z^2(G_1*G_2, \Toro)$.
\end{proposition}

\begin{proof}
It is essentially already known that $L$ is negative definite; indeed, this can be deduced in the same way as in the proof of Proposition \ref{weakheatfreeprod}.
 
It is not difficult to check that $L$ is proper. Moreover, $G$ has polynomial H-growth w.r.t.~$L$. This can be deduced from \cite{Jol}, as was very briefly mentioned in \cite[Example 3.12; 4)]{BeCo1}. For the ease of the reader, we provide some more details below. For $k\in \Naturali \cup \{0\}$, set $A_k=\{ g\in G : L(g) = k\}$. Then we claim that there exist  $K, p>0$ such that $c(A_k) \leq K (1+k)^p$ for all $k\in \Naturali$ (which will show the desired assertion, cf.~\cite[Lemma 3.11]{BeCo1}). By definition of $c(A_k)$, this amounts to showing that there exists some $K,p>0$ such that
\[\|\lambda(f)\|\leq K (1+k)^p \|f\|_2 \] 
whenever $f\in C_c(G)$ has support in $A_k$.
 
We note that, as each $G_j$ has polynomial H-growth w.r.t.~$L_j$, each $G_j$ has Jolissaint's property (RD) w.r.t.~$L_j$ (this follows from \cite[Theorem 3.13, part 1)]{BeCo1}). Thus the assumptions of \cite[Theorem 2.2.2]{Jol} are satisfied. From the proof of \cite[Theorem 2.2.2]{Jol} (cf.~the bottom of page 184) we deduce that there exist $C, s>0$ such that for any nonnegative integers $k, l, m$ satisfying $|k-l| \leq m \leq k+l$, any $f\in C_c(G)$ having support in $A_k$ and any $\xi$ having support in $A_l$, we have
\[ \|(f*\xi)\chi_m\|_2\,\leq \,C\,(1+k)^{s+2}\, \|f\|_2 \, \|\xi\|_2\,,\]
where $\chi_m$ denote the characteristic function of $A_m$. Now, set $r= s+2$. Then, using this observation, we get, as in the proof of \cite[Proposition 1.2.6, $(4) \Rightarrow$ (1)]{Jol}, that there exists $K >0$ (depending on $C$) such that for any $f\in C_c(G)$ having support in $A_k$ and any $\xi\in C_c(G)$, we have 
\[ \|f*\xi\|_2 \, \leq \, K \, (1+k)^{r+1} \, \|f\|_2 \, \|\xi\|_2\,.\]
Since $\lambda(f) \xi = f*\xi$, where $\lambda(f)= \sum_{g\in G} f(g) \lambda(g)$, it follows that $\|\lambda(f)\| \leq K (1+k)^{r+1} \|f\|_2  $ for any $f\in C_c(G)$ with support in $A_k$, which shows the desired assertion (with $p= r+1 = s+3$).

Theorem \ref{subexp} can now be applied to give the conclusion. 
\end{proof}

We note that the results of Jolissaint  we have been using in the proof of Proposition \ref{freeprod} actually hold for an amalgamated free product $G= G_1*_AG_2$ of groups, with $A$ finite, if one further assumes that $A \simeq \{ g \in G_j: L_j(g) = 0\} $ and $L_j$ is bi-$A$-invariant for each $j=1,2$. When the induced length function $L$  on $G$ is bi-$A$-invariant, one can then deduce from  \cite[Theorem 1]{Pic} that $L$ is negative definite on $G$, and $(G, \sigma)$ will have the heat property w.r.t.~$L$.
   
  We also note that one may associate a product 2-cocycle $\sigma_1*\sigma_2 $ on $G_1*G_2$ to each pair $(\sigma_1, \sigma_2)$, where $\sigma_j \in Z^2(G_j, \Toro)$ for $j=1, 2$. Moreover,  $\sigma\in Z^2(G_1*G_2, \Toro)$ arises in this way, up to similarity. We refer e.g.~to \cite{Oml} for details about these facts. 
   
   \begin{example} To illustrate Proposition \ref{freeprod}, let us assume that $G_1$ and $G_2$ are finite and nontrivial. Let then $L_j$ be the trivial length function on $G_j$ given by $L_j (g) = 1$ if $g\in G_j\setminus \{e_j\}$ and $L_j(e_j) = 0$ for $j=1,2$. It is obvious that each $L_j$ is proper and negative definite on $G_j$. Hence the induced length function $L$ on $G_1*G_2$ is negative definite. (This could also be deduced from the discussion in \cite[Section 6]{Boz89}.) We note that $L(g) = n$ whenever  $g= s_{1}s_{2}\ldots s_{n}$, where each $s_{j} \in G_{i_j}$ for some $i_j\in \{1,2\}$, and $i_{j+1}\neq i_j$ for each $j < n$, so $L$ is sometimes called the block length function on $G_1*G_2$.  Since each $G_j$, being finite, has polynomial growth w.r.t.~$L_j$, we can conclude from Proposition \ref{freeprod} that $(G_1*G_2, \sigma_1*\sigma_2)$ has the heat property w.r.t.~$L$ for every choice of $\sigma_1 \in Z^2(G_1, \Toro)$ and $ \sigma_2 \in Z^2(G_2, \Toro)$. 
   
   As a concrete example, for $j=1, 2$,   pick $n_j \in \Naturali, n_j \geq 2$,  set $G_j=\Relativi_{n_j} \times \Relativi_{n_j}$, and  let $\sigma_j\in Z^2(G_j, \Toro)$ be given by \[\sigma_j((p,q), (r,s)) =  e^{i \frac{2\pi}{n_j} ps},\quad p,q,r,s \in \Relativi_{n_j}.\]  It is well-known that $C_r^*(G_j, \sigma_j) \simeq M_{n_j}(\Complessi)$ and $C_r^*(G_1*G_2, \sigma_1*\sigma_2) \simeq M_{n_1}(\Complessi)*M_{n_2}(\Complessi)$.  
\end{example} 

\begin{example}\label{SL2Z} Let $G=SL(2, \Relativi) \simeq  \Relativi_4 *_{\Relativi_2}\Relativi_6$, where \[\Relativi_2\simeq \{ m \in \Relativi_4: m\in \{0, 2\}\} \simeq \{ n \in \Relativi_6: n\in \{0, 3\}\}.\] Define length functions $L_1$ on $\Relativi_4$ and $L_2$ on $\Relativi_6$ by 
\[L_1(m) =\begin{cases} 1 &  \text{if $m=1$ or $3$},\\ 0 &  \text{if $m=0$ or $2$} \end{cases}, \quad L_2(n) =\begin{cases} 1 &  \text{if $n=1,2,4$ or $5$},\\ 0 &  \text{if $n=0$ or $3$} \end{cases}\] 
Then $\Relativi_2= \{ m \in \Relativi_4 : L_1(m) = 0\} = \{ n \in \Relativi_6 : L_2(n) = 0\}$, and $L_1$ and $L_2$ are clearly $\Relativi_2$-invariant.
So adapting the proof given in the case of a free product (Proposition \ref{freeprod}), we get that $G$ has polynomial H-growth w.r.t.~$L$, where $L$ is the  length function on $G$ induced by $L_1$ and $L_2$. As $L_1$ and $L_2$ are both negative definite (since $1-L_j = \chi_{\Relativi_2}$ is positive definite),  $L$ is negative definite (being bi-$\Relativi_2$-invariant). We can therefore conclude from (the extended version of) Proposition \ref{freeprod} that $SL(2, \Relativi)$ has the heat property w.r.t.~$L$. 
\end{example}
 
\begin{example} \label{CS}
If $G$ is amenable and $d$ is proper, then Theorem \ref{subexp} says that $(G, \sigma)$ has the heat property w.r.t.~$d$ whenever $G$ has subexponential growth w.r.t.~$d$. Assume that $G$  is finitely generated and has polynomial growth w.r.t.~some 
(hence any) word length function $L$. We would like to apply Theorem \ref{subexp} in this context, but is not clear under which conditions $L$ can be chosen to be negative definite. However, Cipriani and Sauvageot have shown, cf.~\cite[Theorem 1.1]{CS}, that for any $\delta > \delta_h$ (where $\delta_h$ denotes the homogeneous dimension of $G$),  there exists a proper $d \in ND_0^+(G)$ and some $k_0 \in \Naturali$ such that 
\[ \big|\{ g \in G: d(g) \leq k\}\big| \leq {O}(k^\delta) \quad \text{whenever } k \geq k_0.\]
Thus it follows that $G$ has polynomial growth w.r.t.~$d$, and we can conclude from Theorem \ref{subexp} that $(G, \sigma) $ has the heat property w.r.t.~$d$. 
This said, it is not clear to us whether the Poincar\'e exponent $\delta(d)$ is equal to zero in this situation. If not,  this would imply that $MCF(G, \sigma) \neq \ell^2(G)$.
\end{example}

\begin{remark} In view of Example \ref{CS}, the following problem seems natural.  Assume that $G$ is finitely generated and has subexponential growth, but not polynomial growth (w.r.t.~some word length function), i.e., $G$ has intermediate growth ($G$ could for instance be the first Grigorchuk group \cite{DH}). Then we ask:  

\medskip \noindent \emph{Does  there exists a proper 
$d\in ND_0^+(G)$ such that $G$ has subexponential  growth w.r.t.~$d$ ?}
 
 \medskip   A positive answer would mean that we can apply Theorem \ref{subexp} and obtain that $G$ has the heat property w.r.t~$d$. Note that if one can find some $d\in ND_0^+(G)$ such that $|\cdot|_S \lesssim d$ for some finite generating set $S$, then $d$ will be proper and $G$ will have subexponential growth w.r.t.~$d$. However, when this happens, we will also have that $\delta(|\cdot|_S) = 0$ by Lemma \ref{alg-length} b), which will imply that $\delta(d)=0$, hence proving that $G$ has the heat property w.r.t~$d$ in a more direct way.

 Curiously enough, by a suitable modification of the arguments given in the proof of \cite[Theorem 1.1]{CS}, we are able to show that the answer to the question above is affirmative for any finitely generated group $G$ with  finite symmetric generator set $S$ satisfying that 
\[\big|\{ g \in G: |g|_S \leq k\}\big| \leq e^{Ck^\alpha} \quad \text{for all } k \in \Naturali\]
for some $C>0$ and some $0< \alpha <1/2$. However, the proof breaks down if one assumes instead that $1/2\leq \alpha < 1$. According to some longstanding conjecture, any such group with $0< \alpha <1/2$ must have polynomial growth (see for instance \cite[Section 10]{Gri}). 
\end{remark}

\begin{example} \label{SL2} Consider $G= \Relativi^2\rtimes SL(2, \Relativi)$, which is a group without the Haagerup property and without property (T), cf.~\cite{CCJJV}. Combining Example \ref{SL2Z} with Proposition \ref{semidirect}, we get that $G$ has the weak heat property. We can in fact show that it satisfies a stronger property. 

Indeed, let $\pi:  \Relativi^2\rtimes SL(2, \Relativi) \to SL(2, \Relativi)$ denote the natural quotient map and let $L$ be the negative definite length function on $SL(2, \Relativi)$ constructed in Example \ref{SL2Z}. We can then extend $L$ to a length function $L'$ on $G= \Relativi^2\rtimes SL(2, \Relativi)$ by setting $ L' := L\circ \pi$. This length function is clearly unbounded and negative definite, but not proper. However, since $SL(2, \Relativi)$ has the heat property w.r.t.~$L$,  we can use the proof of Lemma \ref{subgroup} (with $H= SL(2, \Relativi)$ and $d_H=L$) to get that for each $x_0\in C_r^*(SL(2,\Relativi)) \subseteq C_r^*(\Relativi^2\rtimes SL(2, \Relativi))$ and every $t>0$, we have that $M_t^{L'}(x_0) \in CF(\Relativi^2\rtimes SL(2, \Relativi))$, i.e., we have $(0, \infty) \subseteq I_{x_0}^{L'}$. This shows that  $\Relativi^2\rtimes SL(2, \Relativi)$ satisfies (a strong form of) the weak heat property, which could be thought as some kind of ``relative heat property''. 
 
 Besides, we note that $\Relativi^2\rtimes SL(2, \Relativi)$ is $K$-amenable (as follows from \cite[Proposition 3.3 and Corollary 4.2]{JuVal}).
\end{example}

\section{Around the heat equation}
  Throughout this section,  we consider  a countably infinite group $G$ and $\sigma \in Z^2(G, \Toro)$.
\subsection{} 
Let  $ \Complessi(G, \sigma)$ denote the $*$-subalgebra of $C_r^*(G, \sigma)$ generated by $\Lambda_\sigma(G)$, and consider a function $d:G\to [0, \infty)$. Then the assignment
\[x\mapsto \sum_{g\in G}  -d(g)\,\widehat{x}(g) \Lambda_\sigma(g)\]
gives a linear operator $H_d$ from $ \Complessi(G, \sigma)$ into itself. Let $\mathcal{D}$ be a subspace of $C_r^*(G,\sigma)$ containing $\Complessi(G, \sigma)$ (so $\mathcal{D}$ is dense in  $C_r^*(G,\sigma)$), and let   $H^\mathcal{D}_d: \mathcal{D} \to C_r^*(G,\sigma)$ be a  linear operator (usually unbounded) which extends $H_d$.

We will say that a function $u:[0, \infty) \to C_r^*(G, \sigma)$ is a solution of the \emph{heat problem on  $C_r^*(G, \sigma)$ associated to $H^\mathcal{D}_d$ and $x_0\in C_r^*(G, \sigma)$}, whenever $u$ satisfies the following conditions: 
\begin{itemize}
\item $u(0) = x_0$, 
\item $u(t) \in \mathcal{D}$ for every $t >0$,
\item $u$ is differentiable on $(0, \infty)$ and  $u'(t) = H^\mathcal{D}_d(u(t))$ \  for every $t>0$,
\item $\lim_{t\to 0^+} \|u(t) - x_0\| = 0$ (i.e., $u$ is continuous at $t=0$).
\end{itemize}

\subsection{} Let $\Delta$ denote the Laplace operator on $C^\infty(\Toro)$, $f_0  \in C(\Toro)$, and consider the associated heat problem, i.e., 
 \[v(0) = f_0,\, \,v'(t) = \Delta(v(t))\, \text{ for every $ t>0$},\, \text{ and } \lim_{t\to 0^+} \|v(t) - f_0\|_\infty = 0\,.\]
 As is well known (see for example \cite[p.~63--65]{DMcK}), Fourier's original approach can be made rigorous in this case, and the unique solution $v(t)$
 is given by $v(0)=f_0$, and for each $t>0$ by the uniformly convergent series
 \[v(t) = \sum_{m \in \Relativi}  e^{-tm^2} \widehat{f_0}(m) \, e_m\,,\] 
 where $e_m(\theta):= e^{im\theta}$  for $m\in \Relativi $ and $\theta\in [0, 2\pi)$.
 
 Let $\Phi$ denote the canonical isomorphism between  $C(\Toro)$ and $C_r^*(\Relativi)$, mapping $e_m$ to $\lambda(m)$ for each $m \in \Relativi$. Then $\Phi$ maps $C^{\infty}(\Toro)$ onto  $\mathcal{S} = \big\{ x \in C_r^*(\Relativi) : \widehat{x} \text{ is rapidly decreasing}\big\}$, so 
 the heat problem above corresponds to the heat problem on $C_r^*(\Relativi)$ associated to $H^\mathcal{S}_d$ and $x_0:= \Phi(f_0) \in C_r^*(\Relativi)$, where 
 $d:\Relativi\to [0, \infty)$ is (the negative definite function) given by $d(m) = m^2$. The (unique) solution of this problem is then given by
\[u(t) =  \sum_{m \in \Relativi}  e^{-tm^2}\, \widehat{x_0}(m) \lambda(m),\]
the series being convergent in operator norm for every $t>0$. 

\subsection{} 
We recall (see for example \cite{App, BR}) that a family $\{T_t\}_{t \geq 0}$ of bounded linear operators on a Banach space $X$ is called a \emph{$C_0$-semigroup} if it satisfies that
\begin{itemize}
\item[] a)  $T_{s+t} = T_sT_t$ for all $s, t \geq 0$, 
\item[]b) $T_0 = {\rm id}_X$, 
\item[] c) the map $t\mapsto T_t(x)$ from $[0, \infty)$ to $X$ is (norm-) continuous for every $x\in X$. 
\end{itemize}
Note that if a) and b) are satisfied, then 
c) holds if and only if  $\|T_t(x)-x\| \rightarrow 0$ when ${t\to 0^+}$ for every $x \in X$, cf.~\cite[Section 1.2]{App}. Moreover, if $\sup_{t\geq 0} \|T_t\| < \infty$, then 
it suffices to check that this holds on a dense subspace of $X$. 

Note also that c) holds if one only assumes that the map $t\mapsto T_t(x)$ is weakly continuous for every $x\in X$ (see 
\cite[Section 3.1.2]{BR}). 

\medskip 
Assume now that $d \in ND_0^+(G)$. For each $t\geq 0$, we let $M_t^d$ denote the bounded operator on $C_r^*(G, \sigma)$ associated to the multiplier $e^{-td}$.Then, as is well-known, $\{M_t^d\}_{t\geq 0}$ is a $C_0$-semigroup on $C_r^*(G, \sigma)$. Indeed, one easily sees that a) and b) hold. Further, if $x \in \Complessi(G, \sigma)$, then it is elementary to check that $\lim_{t\to 0^+} M_t^d(x) = x$.  As each $M_t^d$ is unital and completely positive,  we have $\|M_t^d\| = M_t^d(I) = 1$ for all $t\geq 0$. Thus it follows that c) also holds.

The general theory of $C_0$-semigroups  gives that
\[\widetilde{\mathcal{D}} :=\Big\{ x \in C_r^*(G,\sigma) : \lim_{t\rightarrow 0^+} \frac{M_t^d(x)-x}{t} \text{ exists in }  C_r^*(G,\sigma) \Big\}\]
is a dense subspace of $C_r^*(G, \sigma)$. Moreover, the generator $H^{\widetilde{\mathcal{D}}}_d:\widetilde{\mathcal{D}}\to C_r^*(G, \sigma)$ of $\{M_t^d\}_{t\geq 0}$, which is defined by 
\[H^{\widetilde{\mathcal{D}}}_d(x) = \lim_{t\rightarrow 0^+} \frac{M_t^d(x)-x}{t} \quad \text{for every } x \in\widetilde{\mathcal{D}},\]
is a closed linear operator, and we have
\begin{itemize}
\item $M_t^d(\widetilde{\mathcal{D}}) \subseteq \widetilde{\mathcal{D}}$ and $M_t^d (H^{\widetilde{\mathcal{D}}}_d (x)) =  H^{\widetilde{\mathcal{D}}}_d (M_t^d(x))$ for all $t \geq 0$ and all $x \in\widetilde{\mathcal{D}}$, 
\item for each $x_0 \in\widetilde{\mathcal{D}}$, the map $u:[0, \infty) \to C_r^*(G, \sigma)$ given by $u(t) := M_t^d(x_0)$ is the (unique) solution of the initial value problem  
\[u'(t) = H^{\widetilde{\mathcal{D}}}_d(u(t)) \text{ for } t>0, \quad u(0) = x_0.\]
\end{itemize}
An alternative description of $\widetilde{\mathcal{D}}$ is that it consists of those $x \in C_r^*(G,\sigma)$ for which there exists a $y\in C_r^*(G, \sigma)$ such that  \[\lim_{t\rightarrow 0^+} \psi\Big(\frac{M_t^d(x)-x}{t}\Big) = \psi(y)\] for all $\psi\in C_r(G, \sigma)^*$, in which case we have $H^{\widetilde{\mathcal{D}}}_d(x)=y$ (cf.~\cite[Corollary 3.1.8]{BR}).

Now, if $x \in \Complessi(G, \sigma)$, so $x=\sum_{g\in F} \widehat{x}(g) \Lambda_\sigma(g)$ for  a finite subset $F$ of $G$, we have 
\[ H^{\widetilde{\mathcal{D}}}_d(x) =  \lim_{t\rightarrow 0^+} \frac{1}{t} \sum_{g\in F}  \big(e^{-td(g)}-1\big)\, \widehat{x}(g) \Lambda_\sigma(g) = - \sum_{g\in F} d(g) \widehat{x}(g) \Lambda_\sigma(g) = H_d(x),\] 
i.e.,  $H^{\widetilde{\mathcal{D}}}_d$ extends $H_d$. Thus the function $u$, defined on $[0, \infty)$, 
is a solution of the heat problem on  $C_r^*(G, \sigma)$ associated to $H^{\widetilde{\mathcal{D}}}_d$ and $x_0$ whenever $x_0$ lies in $\widetilde{\mathcal{D}}$. 

However, if $x_0 \not\in \widetilde{\mathcal{D}}$ and $t>0$, then  it is not clear when $u(t)=M_t^d(x_0)$ belongs to $\widetilde{\mathcal{D}}$. 
We are interested in solving the heat problem for every initial datum $x_0 \in C_r^*(G, \sigma)$. As it is not easy to give a concrete description of the domain 
$\widetilde{\mathcal{D}}$ one would like to find a \emph{core}  for $H^{\widetilde{\mathcal{D}}}_d$ (that is, a subspace $\mathcal{C}$ of $\widetilde{\mathcal{D}}$ such that $H^{\widetilde{\mathcal{D}}}_d$ is the closure of its restriction to $\mathcal{C}$). Since $\{M_t^d\}_{t\geq 0}$ is a $C_0$-semigroup, \cite[Theorem 1.3.18]{App}, or \cite[Corollary 3.1.7]{BR}, tells us that it suffices to find a dense subspace $\mathcal{C}$ of $C_r^*(G, \sigma)$ such that $\mathcal{C}\subseteq\widetilde{\mathcal{D}}$ and  $M_t^d(\mathcal{C}) \subseteq \mathcal{C}$ for all $t> 0$. 
We discuss such a core $\mathcal{C}$ in the next subsection.

\subsection{} \label{CH}
Set \[\mathcal{C}:= \big\{ x\in C_r^*(G, \sigma) : \sum_{g\in G} d(g) \widehat{x}(g) \Lambda_\sigma(g) \, \text{ is convergent}\big\}.\] Clearly, $\mathcal{C}$ is a subspace of $C_r^*(G, \sigma)$ such that $\Complessi(G, \sigma) \subseteq \mathcal{C}$, and  $H^\mathcal{C}_d: \mathcal{C} \to C_r^*(G, \sigma)$ defined by 
\[H^\mathcal{C}_d(x) := -\sum_{g\in G} d(g) \widehat{x}(g) \Lambda_\sigma(g)\]  is a linear map.

We note that $M_t^d(\mathcal{C}) \subseteq \mathcal{C}$ for all $t>0$. Indeed, if $x_0 \in \mathcal{C}$, $t> 0$, and $u(t):=M_t^d(x_0)$, then
\[\sum_{g\in G} d(g)\widehat{u(t)}(g) \Lambda_\sigma(g) =  \sum_{g\in G} e^{-td(g)} \, d(g) \widehat{x_0}(g) \Lambda_\sigma(g)\]
is convergent by Proposition \ref{bdprop}. Thus,  $u(t) \in \mathcal{C}$.

It follows that $\mathcal{C}$ will be a core for $H^{\widetilde{\mathcal{D}}}_d$ whenever $\mathcal{C} \subseteq\widetilde{\mathcal{D}}$. 
We will use this fact in the proof of Proposition \ref{Core}.

\medskip Now, let $x \in \mathcal{C}$ and set $y(t):=  \frac{1}{t}(M_t^d(x)-x)$ for each $t>0$. Define $k: [0, \infty) \to \Reali$ by \[ k(a) = \begin{cases} \frac{1}{a}(1-e^{-a}) & \text{ if } a>0,\\ 1 & \text{ if } a=0.\end{cases}\]
It is easy to see that $0\leq k(a) \leq 1 $ for every $a\geq 0$. Thus, the map $\eta_t: G \to \Reali$ given by $\eta_t(g) := - k(td(g))$  is bounded for every $t>0$. 
Let now $t>0$. Note that 
\begin{equation} \label{Fy} 
\widehat{y(t)}(g)  = \frac{e^{-td(g)} -1}{t} \, \widehat{x}(g) = \eta_t(g) \, d(g) \widehat{x}(g) \quad\text{ for all } g\in G.
\end{equation}  
Further, since $\sum_{g\in G} d(g)\, \widehat{x}(g) \Lambda_\sigma(g)$ is convergent (by definition of $\mathcal{C}$), Proposition \ref{bdprop} gives that the series
\[ \sum_{g\in G} \eta_t(g) \, d(g) \widehat{x}(g)\Lambda_\sigma(g) \, \text{ is also convergent}.\]
As this series is the Fourier series of $y(t)$, cf.~(\ref{Fy}), we get that 
\[y(t) = \sum_{g\in G}  \frac{1}{t}(e^{-td(g)} -1) \, \widehat{x}(g)\Lambda_\sigma(g)\]
the series on the right-hand side being convergent. Thus, we have shown the following lemma.

\begin{lemma} \label{diff} Assume that $x\in \mathcal{C}$. Then for every $t>0$ we have 
\[ \frac{1}{t}(M_t^d(x)-x) = \sum_{g\in G}  \frac{1}{t}(e^{-td(g)} -1) \, \widehat{x}(g)\Lambda_\sigma(g)\]
$($where the series on the right-hand side is convergent$)$.
\end{lemma} 

\begin{proposition}\label{Core}
$\mathcal{C}$ is a core for $H_d^{\widetilde{D}}$, and $H_d^\mathcal{C}$ is the restriction of $H^{\widetilde{\mathcal{D}}}_d$ to $\mathcal{C}$.
\end{proposition} 
\begin{proof} 
Let $x\in \mathcal{C}$ and set $y := - \sum_{g\in G} d(g)\, \widehat{x}(g) \Lambda_\sigma(g)\in CF(G, \sigma)$. To show that $\mathcal{C} \subseteq \widetilde{\mathcal{D}}$, it suffices to show that  
\[\lim_{t\rightarrow 0^+} \psi\Big(\frac{M_t^d(x)-x}{t}\Big) = \psi(y)\] for every $\psi\in C_r(G, \sigma)^*$. Using Lemma \ref{diff}, this amounts to show that
 \begin{equation} \label{deriv} 
 \lim_{t\to 0^+}  \,\sum_{g\in G}  \Big(\frac{1}{t}(e^{-td(g)} -1) + d(g)\Big) \,\widehat{x}(g) \, \psi\big(\Lambda_\sigma(g)\big)=0
  \end{equation}
  for every $\psi\in C_r(G, \sigma)^*$. So let $\psi\in C_r(G, \sigma)^*$ and consider any sequence $\{t_n\}$ in $(0, \infty)$ converging to $0$. Then, for each $n \in \Naturali$, define $f_n: G \to \Complessi$ by 
  \[ f_n(g) =  \Big(\frac{1}{t_n}(e^{-t_nd(g)} -1) + d(g)\Big)  \,\widehat{x}(g)\,\psi\big(\Lambda_\sigma(g)\big)\,.\]
  Since $\lim_{n\to \infty} \frac{1}{t_n}(e^{-t_nd(g)} -1)= -d(g)$, we have that 
  \[\lim_{n\to \infty} f_n(g) = 0.\]
Further, since $-1\leq \eta_{t_n}(g)\leq 0$ for every $g\in G$, we get that 
  \[ |f_n(g)| = \big|(\eta_{t_n}(g) +1) d(g) \,\widehat{x}(g)\,\psi\big(\Lambda_\sigma(g)\big)\big| \leq \, d(g) \,\big|\widehat{x}(g)\psi\big(\Lambda_\sigma(g)\big)\big|\]
  for all $n\in \Naturali$ and $g\in G$.   
  
Now, since $ \sum_{g\in G} d(g)\, \widehat{x}(g) \Lambda_\sigma(g)$ is convergent in $C_r^*(G, \sigma)$, $ \sum_{g\in G} d(g)\, \widehat{x}(g)\, \psi\big(\Lambda_\sigma(g)\big)$ is  convergent in $\Complessi$, and this is equivalent to 
   \[ \sum_{g\in G}  d(g)\,\big| \widehat{x}(g) \psi\big(\Lambda_\sigma(g)\big)\big|\, < \, \infty.\]
Thus, we may invoke the dominated convergence theorem and get that 
  \[  \lim_{n\to \infty} \, \sum_{g\in G} f_n(g) = \sum_{g\in G}\, \lim_{n\to \infty} f_n(g) = 0.\]
 It follows that (\ref{deriv}) holds.
\end{proof}

\subsection{} In this subsection, we let  $x_0 \in C_r^*(G, \sigma)$ and  $u:[0, \infty) \to C_r^*(G, \sigma)$ be the function given by  $u(t):=M_t^d(x_0)$ for each $t\geq 0$.

\begin{lemma} \label{CF-deriv} 
Assume $u'(t)$ exists and belongs to $CF(G, \sigma)$ for some $t>0$. Then we have $u(t) \in \mathcal{C}$ and $u'(t) = H^\mathcal{C}_d(u(t))$.
\end{lemma}

 \begin{proof} Let $g\in G$. Since $u'(t)$ exists in $C_r^*(G, \sigma)$, we have that 
 \begin{align*}\widehat{u'(t)}(g) &= \lim_{h\to 0} \Big[\frac{u(t+h)-u(t)}{h} \delta_e\Big](g) =  \lim_{h\to 0} \Big[\frac{\widehat{M_{t+h}^d(x_0)}-\widehat{M_t^d(x_0)}}{h}\Big](g)\\
&= \lim_{h\to 0}\, \frac{1}{h}\big(e^{-(t+h)d(g)}\,\widehat{x_0}(g) -e^{-td(g)}\,\widehat{x_0}(g)\big) = 
\lim_{h\to 0} \,\frac{1}{h}\big(e^{-hd(g)}-1\big) e^{-td(g)}\,\widehat{x_0}(g) \\
&= -d(g) e^{-td(g)} \,\widehat{x_0}(g).
\end{align*}
Hence, since $u'(t) \in CF(G, \sigma)$ (by assumption), we have that 
\[ u'(t) =  -\sum_{g\in G} d(g) e^{-td(g)} \, \widehat{x_0}(g) \Lambda_\sigma(g) = -\sum_{g\in G} d(g) \, \widehat{u(t)}(g) \Lambda_\sigma(g).\]
Thus, $u(t) \in \mathcal{C}$ and $u'(t) = H^\mathcal{C}_d(u(t))$.
\end{proof}

Using this lemma, we readily get:
\begin{proposition}\label{heatproblem}
The function $u$ is a solution of the heat problem associated to $H^\mathcal{C}_d$ and $x_0$ if and only if $u$ is differentiable on $(0, \infty)$ and $u'(t) \in CF(G, \sigma)$ for all $t>0$.
\end{proposition} 

The following uniqueness result is reminiscent of its classical counterpart, going back to Fourier.  
 \begin{proposition}\label{unique}
Assume that $v:[0, \infty)\to \mathcal{C}$ is a solution of the heat problem associated to $H_d^{\mathcal{C}}$ and $x_0$. Then $v(t) = M_t^d(x_0)$ for every $t\geq 0$.
\end{proposition}
\begin{proof}
Let $t>0$. Since $v(t) \in \mathcal{C}$,  we have $v'(t) = -\sum_{g\in G} d(g) \widehat{v(t)}(g) \Lambda_\sigma(g)$, hence $\widehat{v'(t)}(g)=-d(g)  \widehat{v(t)}(g)$ for every $g\in G$. Now, for each $g\in G$, set $w_g(t):=\widehat{v(t)}(g)$. Then we have
\begin{align*}  \lim_{h\to 0} \frac{w_g(t+h) - w_g(t)}{h}
&=\lim_{h\to 0} \frac{\widehat{v(t+h)}(g) - \widehat{v(t)}(g)}{h}\\
&= \lim_{h\to 0} \,\tau\Big(\frac{1}{h}\,\big(v(t+h)- v(t)\big)\Lambda_\sigma(g)^*\Big)\\
&= \tau\big(v'(t)\Lambda_\sigma(g)^*\big)=\widehat{v'(t)}(g),
\end{align*}
which gives that \[(w_g)'(t) = \widehat{v'(t)}(g) = - d(g)  \widehat{v(t)}(g) = -d(g)\,w_g(t)\] for each $g\in G$. Thus we get that $\widehat{v(t)}(g) = w_g(t)= c_g e^{-td(g)}$ for all $t >0$ for some $c_g \in \Complessi$. As $\lim_{t\to 0^+} v(t) = x_0$, we readily get that $c_g=\lim_{t\to 0^+} \widehat{v(t)}(g)= \widehat{x_0}(g)$. Hence, for any $t>0$, we have
\[\widehat{v(t)}(g) = e^{-td(g)}\,\widehat{x_0}(g) = \widehat{u(t)}(g)\]
for all $g\in G$, which implies that $v(t) = u(t)=M_t^d(x_0)$.
\end{proof}

We now consider the case where $(G, \sigma)$ has the heat property w.r.t.~$d$. 
\begin{lemma} \label{heat-C} Assume that $(G, \sigma)$ has the heat property w.r.t.~$d$. Then $u(t) \in \mathcal{C}$ for all $t>0$. 
\end{lemma}
\begin{proof}
Let $t>0$ and pick $s$ such that $0< s < t$. Then we have
 \begin{align*} \sum_{g\in G}d(g)\widehat{u(t)}(g) \Lambda_\sigma(g) &=  \sum_{g\in G} d(g) e^{-td(g)} \, \widehat{x_0}(g) \Lambda_\sigma(g)\\
 &= \sum_{g\in G} d(g)e^{-(t-s)d(g)} \, e^{-sd(g)}\widehat{x_0}(g) \Lambda_\sigma(g)
 \end{align*}
 As $r\mapsto re^{-ar} $ is a bounded function on $[0, \infty)$ whenever $a >0$,  we get that the function $g \mapsto d(g)e^{-(t-s)d(g)}$ is bounded on $G$. Further, since $e^{-sd} \in MCF(G,\sigma)$ (by  assumption), the series  $\sum_{g\in G} e^{-sd(g)}\widehat{x_0}(g) \Lambda_\sigma(g)$ is convergent. It therefore follows from the expression above  that the series $\sum_{g\in G}d(g)\widehat{u(t)}(g) \Lambda_\sigma(g)$ is convergent, i.e., $u(t) \in \mathcal{C}$.\footnote{A similar argument gives that if $e^{-t_0d} \in MCF(G,\sigma)$ for some $t_0 > 0$, then $u(t) \in \mathcal{C}$ for every $t > t_0$.}
\end{proof}
  
This means that  $H^\mathcal{C}_d(u(t))$ is defined for every $t>0$ whenever $(G, \sigma)$ has the heat property w.r.t.~$d$. 
\begin{theorem} \label{heat-solve} Assume  $(G, \sigma)$ has the heat property w.r.t.~$d$, and let $x_0$ be any element of $C_r^*(G, \sigma)$. Then the function $u$ given by $u(t) = M_t^d(x_0)$ for each $t\geq0$, which has a convergent Fourier series for every $t>0$, is the unique solution of the heat problem associated to $H_d^{\mathcal{C}}$ and $x_0$.  
\end{theorem}
\begin{proof}
Let  $t>0$. Lemma \ref{heat-C} gives that $u(t) \in \mathcal{C}$, hence that  $u(t) \in \widetilde{\mathcal{D}}$ by  Proposition \ref{Core}. Then, by definition of $\widetilde{\mathcal{D}}$ and $H_d^{\widetilde{\mathcal{D}}}$, we get that 
\[u'(t) = \lim_{h\to 0^+} \frac{1}{h}\big(u(t+h)-u(t)\big) = \lim_{h\to 0^+} \, \frac{1}{h}\big(M_h^d(u(t))-u(t)\big) \, \text{ exists in } C_r^*(G,\sigma)\]
and is equal to $H_d^{\widetilde{\mathcal{D}}}(u(t)) $ $= H_d^{\mathcal{C}}(u(t))$. Since $u(0)=x_0$ and  $\|u(t)-x_0\| \to 0$ as $t\to 0^+$, this shows that $u$  is a solution of the heat problem associated to $H_d^{\mathcal{C}}$ and $x_0$. The uniqueness part follows from Proposition \ref{unique}.
\end{proof}

\begin{remark}
Assume that  $(G, \sigma)$ does not have the heat property w.r.t.~$d$, and let $x_0\in CF(G, \sigma)$. Then $u(t)= M_t^d(x_0) \in CF(G, \sigma)$ for every $t>0$ by Proposition \ref{CFG2}, so the proof of Lemma \ref{heat-C} gives that $u(t) \in \mathcal{C}$ for every $t>0$. Moreover, the proof of Theorem \ref{heat-solve} also goes through, and we can conclude that the function $u$ is the unique solution of the heat problem associated to 
$H_d^{\mathcal{C}}$ and $x_0$. However, this may not be true if $x_0 \not\in CF(G, \sigma)$ (cf.~Proposition \ref{T-heat}).
\end{remark} 

Theorem \ref{heat-solve} may be applied to all the examples we have exhibited in the previous section. We illustrate this with a couple of cases. 
 \begin{example} 
Let $d\in ND_0^+(\Relativi^n)$ be  given by $d(m):= m_1^2 + \cdots + m_n^2$, $\Theta\in M_n([0, 2\pi))$ be a skew-symmetric matrix and consider the heat problem on $A_\Theta = C_r^*(\Relativi^n, \sigma_\Theta)$ associated to $H_d^\mathcal{C}$ and $x_0\in A_\Theta$. Then the domain $\mathcal{C}$ clearly contains
\[\Big\{  x \in A_\Theta: \sum_{m\in \Relativi^n} (m_1^2 + \cdots + m_n^2)\, \big|\widehat{x}(m_1,  \ldots, m_n)\big| < \infty\Big\},\] hence also the smooth algebra $A_\Theta^\infty:=\{ x \in A_\Theta: \widehat{x} \text{ is rapidly decreasing on } \Relativi^n\}$,
 and $H_d^\mathcal{C}$ is nothing but the Laplace operator $\Delta =\delta_1^2 +\cdots + \delta_n^2$ when restricted to  $A_\Theta^\infty$. (See \cite{Ros} for a study of the associated Laplace equation and its variants when $n=2$.) Since $\Relativi^n$ has the heat property w.r.t.~$d$, we get that the unique solution of the heat problem is given by $u(0) = x_0$ and 
\[ u(t) =  \sum_{m\in \Relativi^n} e^{-t(m_1^2 + \cdots + m_n^2)}\, \widehat{x_0}(m) \Lambda_{\sigma_\Theta}(m)\,, \quad  t>0,\]
which actually belongs to $A_\Theta^\infty$ for every $t>0$. Here, instead of $d=|\cdot|^2_2$, one could also consider $d=|\cdot|_1$ or  $d=|\cdot|_2$. 
\end{example}

\begin{example} Let $G=\mathbb{F}_k$ ($k\geq 2$) and $L=|\cdot|$ be the canonical word-length on $\mathbb{F}_k$, as in Example \ref{free-gp2}. Then $H:=H_L^\mathcal{C}$ is given by 
\[ H(x) =- \sum_{g\in \mathbb{F}_k} |g| \, \widehat{x}(g) \lambda(g)\]
whenever  $\sum_{g\in \mathbb{F}_k} |g| \, \widehat{x}(g) \lambda(g)$ is convergent. Since $\mathbb{F}_k$ has the heat property w.r.t.~$|\cdot|$, we get that the heat problem  associated to $H$ and $x_0\in C_r^*(\mathbb{F}_k)$  has a unique solution $u$, where $u(t)$ is given for each $t>0$ by the convergent Fourier series in (\ref{free-sol}).
\end{example}

Under some mild assumption on $d$, we can show that if the heat problem associated to $H^\mathcal{C}_d$ has a solution for every choice of initial value in $C_r^*(G, \sigma)$, then $(G, \sigma)$ must have the heat property w.r.t.~$d$. We will use the following lemma, where  \[G^d_0:=\{ g \in G \mid d(g)=0\}.\]
\begin{lemma}\label{d-conv}
Assume $G_0^d$ is finite and $\inf\big\{ d(g): g\in G\setminus G^d_0\big\} >0$. Then $\mathcal{C} \subseteq CF(G, \sigma)$.
\end{lemma}
\begin{proof}
Let $x \in \mathcal{C} $. Since the function $\frac{1}{d}$ is bounded on $G\setminus G^d_0$, we get that 
\[ \sum_{g\in G\setminus G^d_0}  \widehat{x}(g) \Lambda_\sigma(g) = \sum_{g\in G\setminus G^d_0} \frac{1}{d(g)} \,d(g) \widehat{x}(g) \Lambda_\sigma(g)\]
is convergent by Proposition \ref{bdprop}. Thus $\sum_{g\in G}  \widehat{x}(g) \Lambda_\sigma(g)$ is convergent too.
\end{proof} 

\begin{proposition} 
Assume  $G^d_0$ is finite and $\,\inf\big\{ d(g): g\in G\setminus G^d_0\big\} >0$. Suppose that the heat problem associated to $H^\mathcal{C}_d$ and $x_0$ has a solution for every $x_0$ in $C_r^*(G, \sigma)$.  Then $(G, \sigma)$ has the heat property w.r.t.~$d$.
\end{proposition}
\begin{proof}  Let $t>0$. Then for any  $x_0 \in C_r^*(G, \sigma)$, we have $M_t^d(x_0) \in \mathcal{C}$, so $M_t^d(x_0)$ belongs to $CF(G, \sigma)$ by Lemma \ref{d-conv}. 
\end{proof}

On the other hand, the heat problem does not always have a solution:
\begin{proposition} \label{T-heat}
 Assume that $G_0^d$ is finite and $\inf\big\{ d(g): g\in G\setminus G^d_0\big\} >0$. If $G$ has property $(T)$ and $x_0 \in C_r^*(G, \sigma)\setminus CF(G, \sigma)$, then the heat problem associated to $H^\mathcal{C}_d$ and $x_0$ has no solution.
 \end{proposition} 

\begin{proof} Assume  the heat problem associated to $H^\mathcal{C}_d$ and $x_0 \in C_r^*(G, \sigma)$ has a solution $v$. Proposition \ref{unique} gives that $v(t) = M_t^d(x_0)$ for all $t\geq 0$. Moreover, using Lemma \ref{d-conv}, we get $M_t^d(x_0) \in \mathcal{C} \subseteq CF(G, \sigma)$ for every $t>0$.  
Now, if $G$ has property $(T)$, then $d$ is bounded, and Proposition \ref{propT} then implies that $x_0 \in CF(G,\sigma)$.
 \end{proof}

\section{Further questions and comments}
{\bf 5.1}\,
Let $d$ in $ND_0^+(G)$. Proposition \ref{propT} shows that $\epsilon(d)= \infty$ if $G$ has property (T). On the other hand, we have that  $\epsilon(d) = 0$ if and only if $G$ has the heat property w.r.t.~$d$, and this happens for example when $d$ is proper and $G$ has subexponential H-growth w.r.t.~$d$, cf.~Theorem \ref{subexp}. It is not difficult to find examples where $0 < \epsilon(d) < \infty$. For instance, if $G=\Relativi$ and $d(n) = \log(1+|n|)$, then $\sum_{n \in \Relativi} e^{-t d(n)} = \sum_n \frac 1 {(|n|+1)^t}$ is convergent if and only if $t > 1$, so  $\delta(d) =1$. Hence, using Remark \ref{poincare-exp}, we get that $\epsilon(d)=\delta(d)/2 = 1/2$. However, $\Relativi$ has the heat property. It would be interesting to know if there are examples of groups without the heat property such that for some $d \in ND_0^+(G)$ we have  $0 < \epsilon(d) <  \infty$, while $\delta(d) = \infty$.
   
Suppose now $G$ does not have property (T). It may well happen that  $\epsilon(d)=\infty$ for some unbounded $d$ (consider for example $G=\Relativi$ and $d(n) =  \log(1+\log(1+|n|))$). One may then ask: when does it exist some $d\in ND_0^+(G)$ such that $\epsilon(d) < \infty$ ? Clearly, any such group will have the weak heat property. In the case where $G$ is abelian, this amounts to asking when there exists some $d\in ND_0^+(G)$ such that $\delta(d) < \infty$.

\medskip \noindent {\bf 5.2}\, All the groups with the heat property we have met in this article also have the Haagerup property. Does this hold in general ? 

The converse seems also open: does there exist a group with the Haagerup property not having the heat property ? For instance, what about the Thompson's groups $F, T$ and $V$ (which are known to have the Haagerup property \cite{Far}, see also \cite{BJ} for $F$ and $T$) ? Actually, we don't know if every amenable, or even abelian, group has the heat property (or at least the weak heat property).

\medskip \noindent {\bf 5.3}\,   In view of our previous works \cite{BeCo3, BeCo5} on Fourier series and discrete $C^*$-crossed products, it is tempting to guess that some of our results in the present paper might be extended to this more general setting, or even to Fell bundle over discrete groups \cite{Exel}.  

\medskip \noindent {\bf 5.4}\, In this paper, we have considered a function $d:G \to [0, \infty)$ such that $d(e)=0$ and assumed that  $d$  negative definite. One may relax this assumption and require instead that $d$ satisfies the following conditions: 
\begin{itemize}
\item[(i)] $e^{-td}$ is a multiplier of $C_r^*(G, \sigma)$ for every $t>0$;
\item[(ii)] $\sup_{t>0} \|M_t^d\| < \infty$ (where $M_t^d= M_{e^{-td}}$).
\end{itemize}

Then one may for instance say that $(G, \sigma)$ has the heat property w.r.t.~$d$ whenever $e^{-td}$ belongs to $ MCF(G, \sigma)$ for every $t>0$. (This amounts to requiring that $\{e^{-td}\}_{t>0}$ is a bounded Fourier summing net in the sense of \cite{BeCo1}.)    
  
 Let then $G$ be a Gromov hyperbolic group and $d=|\cdot|$ be the word length function on $G$ w.r.t.~some finite set of generators. Then it follows from \cite[Theorem 1]{Oz} (and \cite[Proposition 4.3]{BeCo1} when $\sigma\neq 1$) that conditions (i) and (ii) are satisfied, although $|\cdot|$ is not necessarily negative definite. (In fact, Ozawa shows that each $M_t^d$ is completely bounded and $\sup_{t>0} \|M_t^d\|_{\rm cb} < \infty$. As shown by Mei and de la Salle in \cite{MdlS}, Ozawa's result can be generalized to the case where $d=|\cdot|^r$ for any $r>0$.)
 Moreover, \cite[Theorem 5.12]{BeCo1} gives that $G$ has the heat property w.r.t.~$|\cdot|$.  

 As  there exist Gromov hyperbolic groups having property (T), we see that this leads to a theory where property (T) is no longer an obstruction to heat properties.  On the other hand, one readily checks that  the  proof of Theorem \ref{heat-solve}   goes through whenever $G$  has the heat property w.r.t.~$d$ in this broader sense, hence that the heat problem associated to $H_d^\mathcal{C}$ and $x_0 \in C_r^*(G, \sigma)$ always has a (unique) solution in this case.

\medskip \noindent {\bf 5.5}\, When considering the heat equation on the circle, it is well-known that the assumption on the initial datum $f_0$ can be relaxed, e.g., one can require $f_0$ to be only piecewise continuous (cf.~\cite[Section 10.2]{TW}).  A general set-up including this case is as follows. 

Assume that $d\in ND_0^+(G)$, or more generally,  that $d$ satisfies the same assumptions as in 5.4. Let the domain $\mathcal{C} \subseteq C_r^*(G, \sigma)$ and $H_d^\mathcal{C}: \mathcal{C} \to CF(G, \sigma)$ be as in subsection \ref{CH}. 

A function $u:[0, \infty) \to {\rm vN}(G, \sigma)$ is then said to be a solution of the \emph{heat problem associated to $H^\mathcal{C}_d$ and $x_0\in {\rm vN}(G, \sigma)$} whenever $u$ satisfies the following conditions: 
\begin{itemize}
\item $u(0) = x_0$, 
\item $u(t) \in \mathcal{C}$ for every $t >0$,
\item $u$ is differentiable on $(0, \infty)$ and  $u'(t) = H^\mathcal{C}_d(u(t))$ \  for every $t>0$,
\item $\lim_{t\to 0^+} \|u(t) - x_0\|_\tau = 0$ \quad (where $\|x\|_\tau = \tau(x^*x)^{1/2}$).
\end{itemize}
For each $t\geq 0$, let $\widetilde{M}_t^d: {\rm vN(G, \sigma)}\to {\rm vN(G, \sigma)}$ be given by $\widetilde{M}_t^d := \widetilde{M}_{e^{-td}}$. Let us  say that  $(G, \sigma)$ has the \emph{strong} heat property w.r.t.~$d$ whenever $\widetilde{M}_t^d$ maps ${\rm vN}(G, \sigma)$ into $CF(G, \sigma)$ for every $t>0$. 
 Then it is not difficult to see that our approach for proving Theorem \ref{heat-solve} may be adapted to establish the following result:

\begin{theorem}
 \label{heat-solve2} Assume  that $(G, \sigma)$ has the strong heat property w.r.t.~$d$, and
 let $x_0$ be any element of ${\rm vN}(G, \sigma)$. Then the function $u$ given by $u(t) = \widetilde{M}_t^d(x_0)$ for every  $t\geq0$ is the unique solution of the heat problem associated to $H_d^{\mathcal{C}}$ and $x_0$.  
 \end{theorem}
 Now, one can  check (using results from \cite{BeCo1}) that $(G, \sigma)$ has the strong heat property w.r.t.~$d$ whenever $\delta(d)=0$ or $G$ has subexponential H-growth w.r.t.~$d$. Thus, all examples in subsections \ref{heat-property} and 5.4 of pairs $(G, \sigma)$ having the heat property w.r.t.~$d$ also satisfy the strenghtened version, so Theorem \ref{heat-solve2} applies to all these cases. 
 
 We also note, as pointed out to us by the referee, that one may here also consider $\{\widetilde{M}_t^d\}_{t\geq 0}$ as a CP-semigroup (in the sense of Arveson) on the von Neumann algebra ${\rm vN}(G, \sigma)$, whose generator $L$ may be thought as a noncommutative Laplacian, see \cite[Chapter 7]{Arv}. However, the focus in \cite{Arv} is very different from ours, and the space $\mathcal{C}$, which is a subspace of the domain of $L$, is not discussed in \cite{Arv}.

\medskip \noindent {\bf 5.6}\,  In another direction, it is not difficult to see that our arguments leading to Corollary \ref{propT2} can be adapted to deduce that if $G$ has property (T), then it is not possible to find $d\in ND_0^+(G)$, $x_0 \in {\rm vN}(G, \sigma) \setminus CF(G, \sigma)$ and $t>0$ such that $\widetilde{M}_t^d(x_0)\in CF(G, \sigma)$. It is conceivable that this is always possible whenever $G$ does not have property (T). Of course, this will be the case if one can show that if $G$ does not have property (T), then $(G, \sigma)$ has the weak heat property, cf.~our discussion in subsection \ref{WHP}.

\medskip \noindent {\bf 5.7}\,  
Let $1< p <\infty$. Then each multiplier $\varphi$ on $G$ gives also rise to a bounded map on the $L^p$-space associated to 
${\rm vN}(G)$ with respect to the canonical tracial state $\tau$.  Many authors have studied such maps, with a particular interest in the case where $\varphi_t = e^{-td}$ for some  $d\in ND_0^+(G)$, see e.g.~\cite{JMP, JMP2, JPPP} and references therein. We are not aware of any direct connection between these works and ours. But let us mention that it is observed in \cite[Section 3.3]{JPPP} that property (T) is an obstruction to ultracontractivity estimates.

\bigskip \noindent {\bf Acknowledgements}. Both authors are grateful to the Trond Mohn Foundation (TMS) for partial financial support through the project ``Pure Mathematics in Norway''.

\bigskip
{\parindent=0pt Addresses of the authors:\\

\smallskip Erik B\'edos, Department of Mathematics, University of
Oslo, \\ P.B. 1053 Blindern, N-0316 Oslo, Norway.\\ E-mail: bedos@math.uio.no. \\

\smallskip \noindent
Roberto Conti, Dipartimento SBAI, Sapienza Universit\`a di Roma \\
Via A. Scarpa 16, I-00161 Roma, Italy. \\ E-mail: roberto.conti@sbai.uniroma1.it
\par}

\end{document}